\documentclass{article}
\usepackage{amsmath}
\usepackage{amsfonts}
\usepackage{amssymb}
\usepackage{amsthm}

\usepackage{thmtools,thm-restate}

\def\R{\mathbb{R}}
\providecommand{\abs}[1]{\left\lvert#1\right\rvert}
\providecommand{\set}[1]{\left\{#1\right\}}

\def\cover{{\cal N}}
\def\pack{{\cal P}}
\def\N{{\mathbb N}}
\def\R{{\mathbb R}}

\def\I{{\cal I}}
\def\J{{\cal J}}
\def\S{{\cal S}}
\def\T{{\cal T}}
\def\d{{\rm d}}
\def\define{:=}

\DeclareMathOperator\diameter{diam}

\DeclareMathOperator\gen{gen}

\newtheorem{definition}{Definition}[section]
\newtheorem{lemma}[definition]{Lemma}
\newtheorem{theorem}[definition]{Theorem}
\newtheorem{corollary}[definition]{Corollary}

\newtheorem{proposition}[definition]{Proposition}

\def\words#1{\quad\hbox{#1}\quad}
\def\wwords#1{\qquad\hbox{#1}\qquad}

\begin{document}

  \title{Equi-homogeneity, Assouad Dimension and Non-autonomous Dynamics}

\author{
	Alexander M. Henderson\\
		University of California, Riverside, CA 92521, USA\\
		%email\textup{: \texttt{ahend006@ucr.edu}}
		\texttt{ahend006@ucr.edu}
	\and
	Eric J. Olson\thanks{EJO was partially supported by EPSRC grant EP/G007470/1 while at Warwick on sabbatical leave from University of Nevada Reno.}\\
		University of Nevada, Reno, NV, 89507, USA\\
		\texttt{ejolson@unr.edu}
  	\and
  	James C. Robinson\thanks{JCR was partially supported by an EPSRC Leadership Fellowship, grant EP/G007470/1.}\\
  		University of Warwick, Coventry, UK, CV4 7AL\\
  		\texttt{j.c.robinson@warwick.ac.uk}
  	\and
  	Nicholas Sharples\thanks{NS was partially supported by an EPSRC Career Acceleration Fellowship, grant EP/I004165/1, awarded to Martin Rasmussen, whose support is gratefully acknowledged.}\\
  		Middlesex University, London, UK, NW4 4BT\\
  		\texttt{nicholas.sharples@gmail.com}
}

\maketitle

\begin{abstract}
We show that self-similar sets arising from iterated function systems that satisfy the Moran open-set condition, a canonical class of fractal sets, are `equi-homogeneous'. This is a regularity property that, roughly speaking, means that at each fixed length-scale any two neighbourhoods of the set have covers of approximately equal cardinality.

Self-similar sets are notable in that they are Ahlfors-David regular, which implies that their Assouad and box-counting dimensions coincide. More generally, attractors of non-autonomous iterated functions systems (where maps are allowed to vary between iterations) can have distinct Assouad and box-counting dimensions. Consequently the familiar notion of Ahlfors-David regularity is too strong to be useful in the analysis of this important class of sets, which include generalised Cantor sets and possess different dimensional behaviour at different length-scales.

We further develop the theory of equi-homogeneity showing that it is a weaker property than Ahlfors-David regularity and distinct from any previously defined notion of dimensional equivalence. However, we show that if the upper and lower box-counting dimensions of an equi-homogeneous set are equal and `attained' in a sense we make precise then the lower Assouad, Hausdorff, packing, lower box-counting, upper box-counting and Assouad dimensions coincide.

Our main results provide conditions under which the attractor of a non-autonomous iterated function system is equi-homogeneous and we use this to compute the Assouad dimension of a certain class of these highly non-trivial sets.
\end{abstract}

\section{Introduction}
In this paper we examine in detail the notion of `equi-homogeneity', which is a regularity property of sets introduced by Olson, Robinson \& Sharples \cite{ORS} to study the Assouad dimension of products of `generalised Cantor sets' i.e. Cantor sets in which we allow the portion removed to vary at each stage of the construction. Generalised Cantor sets provide simple examples of equi-homogeneous sets $C\subset \R$ whose lower box-counting, upper box-counting, and Assouad dimensions can take arbitrary values satisfying
\[
\dim_{\rm LB}C\leq \dim_{\rm B}C\leq \dim_{\rm A}C.
\]

In our previous paper \cite{ORS} we demonstrated that a large class of `homogeneous' Moran sets, which include the generalised Cantor sets, are equi-homogeneous. Roughly, the equi-homogeneity property means that at each length-scale the number of balls required in each `local cover' of the set is equal up to some constant factor that is uniform across all length-scales.
For example, it can be shown that the largest local cover of a generalised Cantor set has cardinality at most 6 times that of the smallest local cover at the same length scale (see Olson, Robinson \& Sharples \cite{ORS}).

Generalised Cantor sets are also examples of attractors of non-autonomous iterated function systems. In this paper we extend the regularity results of Olson, Robinson \& Sharples \cite{ORS} to a natural class of attractors of both autonomous and non-autonomous iterated functions systems of contracting similarities. These regularity results are useful as pullback attractors can exhibit dimensionally different behaviour at different length scales and are therefore not, in general, Ahlfors-David regular (discussed in Section 3. See also Heinonen \cite{Heinonen} or Mackay \& Tyson \cite{MikeTyson}).

We discuss how equi-homogeneity relates to other notions of regularity of sets. In particular, we demonstrate that equi-homogeneity is weaker than Ahlfors-David regularity (recalled in Definition \ref{definition - Ahlfors-David}, below), which is the content of Theorem \ref{theorem - ahlfors implies equih}. In this theorem we also prove that the lower Assouad, Hausdorff, packing, lower box-counting, upper box-counting, and Assouad dimensions coincide for Ahlfors-David regular sets. A weaker version of this result (omitting the lower Assouad dimension equality) is regarded as mathematical ``folklore'' (this result is stated in Corollary 3.2 of Farkas and Fraser \cite{FarkasFraser14} and Proposition 2.1 of Tyson \cite{Tyson08}, and essentially follows from (6.6) of Luukkainen \cite{Luuk}).

We further demonstrate that the equi-homogeneity property is distinct from any previously defined notion of dimensional equivalence: the generalised Cantor sets provide examples of equi-homogeneous sets with unequal dimensions (see Olson, Robinson and Sharples \cite{ORS}). Conversely, in Proposition \ref{proposition - equal dimension not equih} we give an example of a set that is not equi-homogeneous yet has coinciding dimensions.

These results establish equi-homogeneity as a widely applicable and useful notion of regularity for fractals and particularly for attractors of iterated functions systems. In contrast, other notions of regularity such as Ahlfors-David regularity or dimensional equality are too restrictive in these contexts as even simple examples of generalised Cantor sets may not have these regularity properties.
 
This paper is organised as follows: in the remainder of the introduction we recall the definitions and properties of pullback attractors and generalised Cantor sets required for our main results, which we summarise at the end of this section.
In Section 2 we recall the notions of dimension and regularity that we will use in the remainder. In Section 3 we define equi-homogeneity and discuss the relationship between various notions of dimension, equi-homogeneity, and other regularity properties of sets. In Section 4 we establish existence and uniqueness results for pullback attractors of non-autonomous iterated function systems before proving our main results that a large class of these pullback attractors are equi-homogeneous.

\subsection{Pullback attractors}
Pullback attractors (see Carvalho, Langa \& Robinson~\cite{Carvalho2013}, Cheban et al. \cite{Cheban2002}, Kloeden \cite{Kloedendifference}, \cite{Kloedensemidynamical}, Kloeden \& Rasmussen \cite{KR}, Kloeden \& Stonier \cite{KloedenStonier1998}, Schamlfu\ss\ \cite{Schmalfuss1992}, for example) were introduced to characterise the possible states at time $t$ of a non-autonomous dissipative continuous dynamical system once all previous initial conditions have been forgotten infinitely far in the past.
Given an initial value problem of the form
\[
	\frac{\d u}{\d t} = f(t,u),\qquad
	u(t_0)=u_0,
\]
define a semi-process $S(t,t_0)$ for $t\ge t_0$ that maps initial values $u_0$ to their subsequent time evolution by
\[
S(t,t_0)(u_0)=u(t)\wwords{for} t\ge t_0.
\]
The pullback attractor is the unique collection of uniformly bounded compact sets $A^t$ for $t\in\R$ such that $A^t=S(t,t_0)A^{t_0}$ for all $t\ge t_0$ and
\[
\rho_H(S(t,t_0)(B),A^t)\to 0\wwords{as} t_0\to-\infty
\]
for all bounded sets $B$. Here $\rho_H(X,Y)$ is the Hausdorff semi-distance
\[
	\rho_H(X,Y)=\sup\limits_{x\in X} \inf\limits_{y\in Y} |x-y|.
\]

It is straightforward to adapt the idea of a pullback attractor to study iterated function systems whose maps change at each step in the iteration. A further generalisation, in which the compositions of maps in a iterated function system are indexed by an infinite tree, arises, for example, from the use of the squeezing property to estimate the upper box-counting dimension of the global attractor of autonomous dissipative continuous dynamical system system (see Eden et al.\ \cite{EFNT}).
In order to keep our notation simple and our presentation self-contained we do not consider iterated function systems indexed by trees here.

To set notation for the rest of this paper we now describe our non-autonomous iterated function systems.
For each $i\in\N$ let $f_i\colon\R^d\to\R^d$ be a contraction with ratio $\sigma_i\in(0,1)$. Thus
\begin{equation}\label{contract}
	|f_i(x)-f_i(y)|\le \sigma_i |x-y|
\wwords{for all} i\in\N.
\end{equation}
We say $f_i$ is a similarity when the above inequality is, in fact, an equality.
For each $k\in\N$ let $\I_k\subset\N$ be an index set with ${\rm card}\left(\I_{k}\right)<\infty$.

Given $B\subseteq\R^d$ and $k<l$ define
\[
	\S^{k,l}(B)=\S^k\circ\ldots\circ\S^{l-1}(B)
\wwords{and}
	\S^{k,k}(B)=B
\]
where $\S^k(B)=\bigcup_{i\in\I_{k+1}} f_i(B).$
We note that $\S$ has the process structure
\[
	\S^{k,l}\circ\S^{l,m}=\S^{k,m}
\wwords{for}
	k\le l\le m.
\]
In particular, $\S^{k,l}$ is the discrete-time analogue of the continuous process operator $S(t,t_0)$, with the identification $t=-k$ and $t_0=-l$.

Since pullback attractors are obtained in the limit $t_0\to-\infty$, the switching of signs between $l$ and $t_0$ allows us the convenience of working with positive indices throughout.
This change of sign also results in a notation that is more consistent with 
the natural notation for autonomous iterated function systems, where the functions are the same at every step.

\begin{definition}\label{definition - pullback attractor}
A pullback attractor of a non-autonomous iterated function system is a collection of sets $\set{F^k}$ for $k\in\N_0=\N\cup\{0\}$ such that
\begin{enumerate}
\item each $F^k$ is compact and uniformly bounded, \label{pullback attractor property 1}
\item the collection $\set{F^k}$ is invariant, in the sense that $ F^k=\S^k( F^{k+1})$ \label{pullback attractor property 2}
holds, and
\item $\rho_H(\S^{k,l}(B), F^k)\to 0$ as $l\to\infty$ for every
bounded set $B\subset\R^d$. \label{pullback attractor property 3}
\end{enumerate}
\end{definition}
It immediately follows from the definition that a pullback attractor, if it exists, is unique (Theorem \ref{unique}).

We note that the requirement that the sets $\{F_k\}$ are uniformly bounded is not normally part of the definition of the continuous-time pullback attractor (see Carvalho, Langa \& Robinson \cite{Carvalho2013}), but in our setting this restriction is both natural and convenient. Without such an assumption the pullback attractor need not attract itself, and then an additional property (such as minimality) is required to ensure uniqueness. Even within the continuous-time setting attractors that are uniformly bounded `in the past' (for all $t\le t_0$ for each $t_0\in\R$) are convenient to avoid various possible pathologies, and this corresponds to uniform boundedness for our iterated function systems which are only defined for indices that correspond to $t\le0$.

Reasonable pullback attractors result when we impose some separation property on the iterated function system (for the autonomous case see, for example, Falconer \cite{BkFalconer14} pp. 139, or Hutchinson \cite{Hutch}). The following condition is a natural generalisation of the familiar Moran open-set condition (see Section \ref{IFS1}).

\begin{definition}\label{definition - generalised MOSC}
A non-autonomous iterated function system satisfies the \emph{generalised Moran open-set condition} if there exists a uniformly bounded sequence of non-empty open sets $U^k\subset\R^d$ for $k\in\N_0$ such
that 
\begin{enumerate}
\item[(i)] $S^k(U^{k+1})\subseteq U^k$;
\item[(ii)] $f_i(U^k)\cap f_j(U^k)=\emptyset$ for $i,j\in\I_k$ such that $i\neq j$; and
\item[(iii)] $\lambda(U^k)\ge \epsilon_0>0$ for all $k\in\N_0$, where $\lambda$ is the $d$-dimensional Lebesgue measure.
\end{enumerate}
\end{definition}
Our main results establish conditions under which pullback attractors satisfying the generalised Moran open-set condition are equi-homogeneous.

\subsection{Generalised Cantor sets}\label{GCS}

The generalised Cantor sets studied in Robinson \& Sharples \cite{RobinsonSharples13RAEX} and Olson, Robinson \& Sharples \cite{ORS} are illustrative examples of pullback attractors.
These sets are defined as follows.
For $\lambda \in (0, 1/2)$ let the application of ${\rm gen}_\lambda$ to a disjoint set of compact intervals be the procedure in which the open middle $1-2\lambda$ proportion of each interval is removed.
Given $c_n\in (0,1/2)$ for all $n\in\N$, the generalised Cantor set $C$ generated by $\set{c_{n}}$ is given by
\[
C={\textstyle \bigcap_{n=1}^\infty} C_n
\words{where}
	C_{n+1}={\rm gen}_{c_{n+1}} C_n
\words{and}
	C_0=[0,1].
\]

By repeatedly taking the left shift of the sequence $\set{c_{n}}$ we produce a countable family of generalised Cantor sets: for each $k\in\mathbb{N}_{0}$ the set $C^{k}$ is the generalised Cantor set generated by the sequence $\set{c_{n+k}}_{n\in\mathbb{N}}$. This family of Cantor sets is a straightforward example of a pullback attractor.

\begin{lemma}
Let $\set{c_{n}}_{n\in\mathbb{N}}$ be a sequence with $c_{n}\in\left(0,1/2\right)$. The collection of generalised Cantor sets $\set{C^{k}}$ for $k\in\mathbb{N}_{0}$ is the pullback attractor of the non-autonomous iterated function system given by 
\begin{equation}\label{gencant}
	f_{2k-1}(x)=c_k x,\quad f_{2k}(x)=c_k x + 1-c_k
\end{equation}
with $\I_k=\{2k-1,2k\}$ for each $k\in\N$.
\end{lemma}
\begin{proof}
Clearly each $C^k$ is compact and uniformly bounded so property \ref{pullback attractor property 1} is satisfied.

Next, writing $C^{k}_{1}=\left[0,1\right]$ and $C^k_{n+1}={\rm gen}_{c_{n+k}} C^k_n$ it is not difficult to see that
\begin{align}\label{prepend}
C^{k}_{n}=\S^{k,k+n}\left(\left[0,1\right]\right)=\S^{k}\circ\ldots\circ \S^{k+n-1}\left(\left[0,1\right]\right).
\end{align}

Consequently,
\begin{align*}
	\S^k C^{k+1}
	=\bigcap_{n=1}^\infty \S^k \circ \S^{k+1,k+n}([0,1])
	=\bigcap_{n=1}^\infty \S^{k,k+n}([0,1])
	=\bigcap_{n=1}^\infty C_{n+1}^{k} = C^{k}
\end{align*}
shows that $C^k$ is invariant, hence property \ref{pullback attractor property 2} is satisfied.
Finally, let $B$ be any bounded set. An argument from Section \ref{non-auto-sec} shows that
\[
\rho_H(S^{k,l}(B),S^{k,l}([0,1]))\le 2^{k-l}\rho_H(B,[0,1])
\]
(this is a consequence of \eqref{lcontract} with $\sigma^*=1/2$) which, coupled with the fact that
\[
\rho_H(C^k_n,C^k)\to 0 \quad \text{as}\quad n\to\infty
\]
yields
\begin{align*}
	\rho_H\big(S^{k,l}(B),C^k\big)
	&\le
	\rho_H\big(S^{k,l}(B),S^{k,l}([0,1])\big)
	+\rho_H\big(S^{k,l}([0,1]),C^k\big)\\
	&\le
	2^{k-l}\rho_H(B,[0,1])
	+\rho_H(C^k_{l-k},C^k)\to 0
\end{align*}%
as $l\to\infty$.  Therefore, property \ref{pullback attractor property 3} is satisfied. Consequently $\set{C^{k}}$ is a pullback attractor of the above iterated function system, which by Theorem \ref{unique} is unique.
\end{proof}

We remark that from \eqref{prepend} we see that applying the procedure $\gen_{c_{n}}$ to the intervals $C_{n}$ is equivalent to replacing the chain of maps $\S^{0,n}$ acting on $\left[0,1\right]$ by $\S^{0,n}\circ \S^{n} = \S^{0,n+1}$, i.e. subsequent levels of the Cantor set construction are obtained by prepending maps in the iterated function system.

\subsection{Summary of main results: Equi-homogeneity of attractors for iterated function systems}
It is well known that self-similar sets (i.e. attractors of autonomous iterated function systems) that satisfy the Moran open-set condition are Ahlfors-David regular (see, for example, Theorem 1(i) Section 5.3 of Hutchinson \cite{Hutch}). By Theorem \ref{theorem - ahlfors implies equih} it follows that such sets are equi-homogeneous.

\begin{restatable}{theorem}{mainzero}
\label{theorem - self-similar}
Let $F$ be the attractor of an autonomous iterated function system of similarities. If $F$ satisfies the Moran open-set condition then $F$ is equi-homogeneous.
\end{restatable}

In Section \ref{non-auto-sec} we extend this analysis to pullback attractors of non-autonomous iterated function systems. Unlike the autonomous case these attractors are typically not Ahflors-David regular; however, they are equi-homogeneous under some mild assumptions. After proving the existence and uniqueness of pullback attractors for iterated function systems of arbitrary contractions (essentially following an argument of Hutchinson \cite{Hutch}) we restrict our attention to iterated function systems of contracting similarities.

Our first result concerning such non-autonomous iterated function systems requires the contraction ratios to coincide at each stage of the iteration. 

\begin{restatable}{theorem}{mainone}
\label{theorem - homogeneous case}
If $\set{F^{k}}$ is the pullback attractor of a non-autonomous iterated function system of similarities satisfying the Moran open-set condition, $\sigma_{i}=c_{k}$ for all $i\in\I_{k}$ and all $k\in\mathbb{N}$, and there exists an $N>0$ such that ${\rm card}(\I_k) \leq N$ for $k\in\N$ then each set $F^{k}$ is equi-homogeneous.
\end{restatable}

We also note (and prove in Section \ref{section - equi-homogeneity and assouad}) that while the resulting pullback attractors are equi-homogeneous, under certain choices for the sequence $c_k$, their Assouad and upper box-counting dimensions are not equal.

Next we turn to the case where the contracting ratios $\sigma_i$ may be different for indices within each index set $\I_k$.
In this case we require some uniformity in the contraction ratios:

\begin{restatable}{theorem}{mainthree}
\label{theorem - general}
Let $\set{F^{k}}$ be the pullback attractor of a non-autonomous iterated function system of similarities that satisfies the Moran open-set condition.
If
\[
\inf\{\,\sigma_i:i\in\N\,\}=\sigma_*>0
\]
and there exists an $s>0$ such that
\begin{equation}\label{hippo}
\sum_{i\in\I_k} \sigma_i^s=1 \words{for all} i\in\I_k
\end{equation}
then each set $F^k$ is equi-homogeneous with $\dim_{\rm A} F^k = s$.
\end{restatable}

Finally, we note that the hypothesis \eqref{hippo} can be weakened a little, for which we introduce the notation 
\begin{align*}
\J^{k,l}&=\I_{k+1}\times\cdots\times\I_l &\text{for}\ k<l.
\end{align*}

\begin{restatable}{theorem}{mainfour}\label{theorem - general averaged}
Let $\set{F_{k}}$ be the pullback attractor of a non-autonomous iterated function system of similarities that satisfies the Moran open-set condition. If
\[
\inf\{\,\sigma_i:i\in\N\,\}=\sigma_*>0
\]
and there exist constants $s>0$, $n_0\in\N$, and $L>0$ such that
\begin{equation}\label{ahippo}
L^{-1}\le {\sum_{\alpha\in\J^{k,k+n}}} \sigma_\alpha^s\le L \words{for all} k\in\mathbb{N}\words{and} n\ge n_0
\end{equation}
then each set $F^k$ is equi-homogeneous with $\dim_{\rm A} F^k=s$.
\end{restatable}
Intuitively \eqref{ahippo} implies that \eqref{hippo} holds when averaged over long enough sequences of iterations.

Note that the existence of generalised Cantor sets with unequal upper box-counting and Assouad dimensions (see Proposition \ref{proposition - equihom distinct dimension}) indicates a problem with the proof in Li \cite{Li2013} as these sets are examples of Moran sets (see Wen \cite{Wen2001}) with contraction ratios bounded below by some positive number. The addition of the assumption \eqref{hippo} (rewritten in the framework of Moran sets) is sufficient to overcome these problems and conclude that the upper box-counting and Assouad dimensions are equal. Further, with this assumption it can be shown that the Moran sets considered in Li \cite{Li2013} are equi-homogeneous (see Olson, Robinson \& Sharples \cite{ORSMoransets}).

\section{Dimension and regularity}\label{section - dimension and regularity}
This section recalls the definitions and some facts about various notions of dimension and regularity of sets in an arbitrary metric space $\left(X,d_{X}\right)$.

We denote by $B_{\delta}(x)=\set{y\in X : d_X(x,y) \leq \delta}$ the closed ball of radius $\delta$ with centre $x\in X$ and
for brevity we refer to closed balls of radius $\delta$ as $\delta$-balls.
For a set $F\subset X$ and each $\delta>0$ we denote  by $\cover(F,\delta)$ the minimum number of $\delta$-balls with centres in $F$ such that $F$ is contained in their union.
A set $F\subset X$ is said to be \textit{totally bounded} if for all $\delta>0$ the quantity $\cover(F,\delta)<\infty$, which is to say that $F$ can be covered by finitely many balls of any radius. We say that a metric space $X$ is \textit{locally totally bounded} if every ball in $X$ is totally bounded.

\subsection{Box-counting and Assouad dimensions}
We begin by recalling the definitions of the box-counting and Assouad dimensions.
\begin{definition}\label{definition - box-counting dimensions}
  For a totally bounded set $F\subset X$ the upper and lower
  box-counting dimensions are defined by
  \begin{align}
	\dim_{\rm B}F&=\limsup_{\delta\to 0+}
    \frac{\log \cover(F,\delta)}{-\log \delta}\\
    \text{and}\qquad
    \dim_{\rm LB}F&=\liminf_{\delta\to 0+}
		\frac{\log \cover(F,\delta)}{-\log \delta},
  \end{align}
  respectively.
\end{definition}

The box-counting dimensions essentially capture the exponent $s\in\R^{+}$ for which $\cover(F,\delta)\sim \delta^{-s}$. More precisely, it follows from Definition \ref{definition - box-counting dimensions} that for all $\varepsilon>0$ and all $\delta_{0}>0$ there exists a constant $C\geq 1$ such that
\begin{align}
  C^{-1}\delta^{-\dim_{\rm LB}F+\varepsilon} &\leq \cover(F,\delta)\leq
  C\delta^{-\dim_{\rm B}F-\varepsilon} & \forall\
  0<\delta\leq\delta_0.\label{box-counting growth bounds}
\end{align}

For some bounded sets $F$ the bounds \eqref{box-counting growth bounds} also hold for $\varepsilon=0$ giving precise control of the growth of $\cover(F,\delta)$. In Olson, Robinson and Sharples \cite{ORS} we distinguish this class of sets and say that they `attain' their box-counting dimensions.

A useful quantity for proving lower bounds is $\pack(F,\delta)$, the maximum number of disjoint $\delta$-balls with centres in $F$, which is related to the minimum cover by centred balls by
\begin{align}\label{geometric inequalities}
\cover(F,2\delta)\leq \pack(F,2\delta)\leq \cover(F,\delta)
\end{align}
(see, for example, `Equivalent definitions' 2.1 in Falconer \cite{BkFalconer14} or Lemma 2.1 in Robinson \& Sharples \cite{RobinsonSharples13RAEX}). In light of the inequalities \eqref{geometric inequalities}, replacing $\cover(F,\delta)$ with $\pack(F,\delta)$ in Definition \ref{definition - box-counting dimensions} gives an equivalent formulation of the box-counting dimensions.

The Assouad dimension is a less familiar notion of dimension, in which we are concerned with `local' coverings of a set $F$: for more details see Assouad \cite{Assouad}, Bouligand \cite{Bouligand}, Luukkainen \cite{Luuk} Olson \cite{EJO}, or Robinson \cite{JCR}.

\begin{definition}\label{definition - Assouad dimension}
The Assouad dimension of a set $F\subset X$ is the infimum over all $s\in \R^{+}$ such that for all $\delta_{0}>0$ there exists a constant $C>0$ for which
\begin{align}
\sup_{x\in F} \cover\left(B_{\delta}\left(x\right)\cap F,\rho\right) \leq C\left(\delta/\rho\right)^{s}\quad \forall\ \delta,\rho \quad \text{with} \quad 0<\rho<\delta\leq\delta_{0}.\label{Assouad scale}
\end{align}
\end{definition}

The lower Assouad dimension, also called the minimal dimension, complements the Assouad dimension with a lower bound on the scaling of local covers (see, for example, Larman \cite{Larman} or Fraser \cite{Fraser2013}).
\begin{definition}\label{definition - lower Assouad dimension}
The lower Assouad dimension of a set $F\subset X$ is the supremum over all $s\in\R^{+}$ such that for all $\delta_{0}>0$ there exists a constant $C>0$ for which
\begin{align*}
\inf_{x\in F} \cover\left(B_{\delta}\left(x\right)\cap F,\rho\right) \geq C\left(\delta/\rho\right)^{s}\quad \forall\ \delta,\rho \quad \text{with} \quad 0<\rho<\delta\leq\delta_{0}
\end{align*}
\end{definition}

Analogous to the box-counting dimensions, the Assouad and lower Assouad dimensions essentially capture the exponent $s\in \R^{+}$ for which $\cover\left(B_{\delta}\left(x\right)\cap F,\rho\right) \sim \left(\delta/r\right)^{s}$. More precisely, it follows from Definitions \ref{definition - Assouad dimension} and \ref{definition - lower Assouad dimension} that for all $\varepsilon >0$ and all $\delta_{0}>0$ there exists a constant $C\geq 1$ such that
\begin{equation}
C^{-1}\left(\delta/\rho\right)^{\dim_{\rm LA}F -\varepsilon} \leq \cover\left(B_{\delta}\left(x\right)\cap F,\rho\right) \leq C\left(\delta/\rho\right)^{\dim_{\rm A}F + \varepsilon}
\end{equation}
for all $x\in F$ and all $\delta,\rho$ with $0<\rho<\delta\leq \delta_{0}$

Minimally for the Assouad and lower Assouad dimensions to be defined we require that each intersection $B_{\delta}(x)\cap F$ is totally bounded. This trivially holds if $X$ is a locally totally bounded space, which is to say that every ball $B_{\delta}(x)\subset X$ is totally bounded (for example, in Euclidean space $X=\R^{n}$).

In this case Definition \ref{definition - Assouad dimension} is equivalent to the simpler formulation of taking the infimum over all $s\in\R^{+}$ for which there exist constants $\delta_{0}>0$ and $C>0$ such that \eqref{Assouad scale} holds, which is easier to check. A further useful formulation can be made if $F$ is itself totally bounded. In this case Definition \ref{definition - Assouad dimension} is equivalent to taking the infimum over $s\in \R^{+}$ for which there exists a constant $C>0$ such that \eqref{Assouad scale} holds for all $\delta,\rho$ with $0<\rho<\delta<\infty$. The lower Assouad dimension has similar equivalent formulations in these cases.
In this paper we study the Assouad dimensions in arbitrary metric spaces, requiring the full generality of Definition \ref{definition - Assouad dimension}, before focusing on attractors of iterated function systems in Euclidean space, which is locally totally bounded.

The following technical lemma gives a relationship between the minimal size of covers of the set $B_{\delta}(x)\cap F$ for different length-scales, which we will use in many of the subsequent proofs.

\begin{lemma}\label{lemma - refinement}
  Let $F\subset X$. For all $\delta,\rho,r>0$ and each $x\in F$
  \begin{align}\label{refinement}
    \cover(B_{\delta}(x)\cap F,\rho)&\leq
    \cover(B_{\delta}(x)\cap F,r) \sup_{x\in F}
    \cover(B_{r}(x)\cap F,\rho)
  \end{align}
\end{lemma}
\begin{proof}
  The only non-trivial case occurs when $\rho<r<\delta$.
 
 If $M:=\cover(B_{\delta}(x)\cap F,r)=\infty$ then there is nothing to prove. Assume that $M<\infty$ and let $x_{1},\ldots, x_{M}\in F$ be the centres of the $r$-balls $B_{r}(x_{j})$ that cover $B_{\delta}(x)\cap F$. Clearly
  \begin{align*}
    B_{\delta}(x)\cap F &\subset \bigcup_{j=1}^{M}
    B_{r}(x_{j})\cap F\\
    \text{so}\qquad \cover(B_{\delta}(x)\cap F,\rho)
	&\leq \sum_{j=1}^{M} \cover(B_{r}(x_{j})\cap F,\rho)\\
    &\leq M \sup_{x\in F} \cover(B_{r}(x)\cap F,\rho)
  \end{align*}
  which is precisely \eqref{refinement}.
\end{proof}

It is known that for a totally bounded set $F\subset X$ the four notions of dimension that we have now introduced satisfy
\begin{align}\label{dimension inequalities 1}
  \dim_{\rm LA}F\leq \dim_{\rm LB}F\leq\dim_{\rm B}F\leq\dim_{\rm A}F
\end{align}
(see, for example, Lemma 9.6 in Robinson \cite{JCR}, Fraser \cite{Fraser2013} or Theorem A.5 in Luukkainen \cite{Luuk}).
Further, if $F$ is compact then
\begin{align}\label{dimension inequalities 2}
\dim_{\rm LA}F \leq \dim_{\rm H} F \leq \dim_{\rm P} F \leq \dim_{\rm B} F \leq \dim_{\rm A} F
\end{align}
(see Larman \cite{Larman}) where $\dim_{\rm H}$ and $\dim_{\rm P}$ are the familiar Hausdorff and packing dimensions (see, for example, Falconer \cite{BkFalconer14}).

The following simple example of a compact countable subset of the real line illustrates that the inequalities involving the Assouad dimensions in \eqref{dimension inequalities 1} can be strict:

\begin{proposition}\label{proposition - different scaling}
 For each $\alpha>0$ the set $F_{\alpha}\define \set{n^{-\alpha}}_{n\in\mathbb{N}}\cup\set{0}$ satisfies
\begin{align*}
\dim_{\rm LA}F_{\alpha} &=0\\
\dim_{\rm LB}F_{\alpha}=\dim_{\rm B}F_{\alpha}&= (1+\alpha)^{-1}
\intertext{and}
\dim_{\rm A} F_{\alpha} &= 1.
\end{align*}
\end{proposition}

\begin{proof}
See Robinson \cite{JCR} Example 13.4 or Falconer \cite{BkFalconer14} Example 2.7 for the derivation of the box-counting dimensions of $F_{\alpha}$.

For the Assouad dimension consider for each $k\in\mathbb{N}$ the lengths $\delta_{k}=k^{-\alpha}$ and $\rho_{k}=\left(2k\right)^{-\alpha}-\left(2k+1\right)^{-\alpha}$. As $k^{\alpha} < \left(2k\right)^{\alpha}$ it is clear that $0<\rho_{k}<\delta_{k}$. Applying the mean value theorem to the function $f\left(x\right)=x^{-\alpha}$ we obtain
\begin{align}
%\frac{f(2k)-f(2k+1)}{2k-(2k+1)} &\leq \max_{2k \leq x\leq 2k+1}f^{\prime}(x)\notag
%\intertext{hence}
\alpha \left(2k+1\right)^{-\alpha-1} \leq \rho_{k} &\leq \alpha \left(2k\right)^{-\alpha-1}\notag
\intertext{from which it follows that}
\frac{1}{\alpha}2^{\alpha+1}k\leq \delta_{k}/\rho_{k} \leq \frac{k^{-\alpha}}{\alpha\left(2k+1\right)^{-\alpha-1}} &= \frac{1}{\alpha}\left(\frac{k}{2k+1}\right)^{-\alpha-1}k \leq \frac{1}{\alpha} 3^{\alpha+1} k \label{example ratio bound}
\end{align}
as $\left(2k+1\right)/k \leq 3$ for all $k\in\mathbb{N}$.

Next, observe that $B_{\delta_{k}}\left(0\right)\cap F_{\alpha} = \set{n^{-\alpha}}_{n \geq k+1}$ and that the distance from any of the $k-1$ points in $F_{\alpha,k}=\set{n^{-\alpha} \mid k+1 \leq n \leq 2k-1}$ to any other point of $F_{\alpha}$ is greater that $\rho_{k}$. Consequently, any covering of $B_{\delta_{k}}\left(0\right)\cap F_{\alpha}$ by $\rho_{k}$ balls with centres in $F_{\alpha}$ requires at least $k-1$ balls for the elements of $F_{\alpha}^{k}$ plus at least one ball for the remaining points, hence from \eqref{example ratio bound}
\begin{align}
N\left(B_{\delta_{k}}\left(0\right) \cap F_{\alpha}, \rho_k \right) &\geq k
\geq \alpha 3^{-\alpha -1} \left(\delta_{k}/\rho_{k}\right) & \forall k\in\mathbb{N}.\label{example Assouad lower bound}
\end{align}

Next, suppose for a contradiction that $\dim_{\rm A}F_{\alpha}<1$, so there exist positive constants $\varepsilon,\delta_{0}$ and $C$ such that
\begin{align*}
N\left(B_{\delta}\left(0\right)\cap F_{\alpha},\rho\right) &\leq C \left(\delta/\rho\right)^{1-\varepsilon} & 0<\rho<\delta\leq \delta_{0}.
\end{align*}
As $\delta_{k} < \delta_{0}$ for all $k\in\mathbb{N}$ sufficiently large it follows from \eqref{example Assouad lower bound} that 
\begin{align*}
\alpha 3^{-\alpha -1} \left(\delta_{k}/\rho_{k}\right) \leq N\left(B_{\delta_{k}}\left(0\right) \cap F_{\alpha}, \rho_k \right) \leq C \left(\delta_{k}/\rho_{k}\right)^{1-\varepsilon}
\end{align*}
hence $\left(\delta_{k}/\rho_{k}\right)^{\varepsilon} \leq C\alpha^{-1} 3^{\alpha+1}$ which is a contradiction as $\delta_{k}/\rho_{k}$ is unbounded from \eqref{example ratio bound}.
We conclude that $\dim_{\rm A}F_{\alpha}=1$, as subsets of the real line have Assouad dimension at most $1$ (see Luukkainen \cite{Luuk}).

For the lower Assouad dimension observe that $1\in F_{\alpha}$ is an isolated point so
  \[
  \inf_{x\in F_{\alpha}} \cover(B_{\delta}(x)\cap
    F_{\alpha},\rho) = 1
  \]
  for all $\delta,\rho$ with $0<\rho<\delta<1-2^{-\alpha}$ as $B_{\delta}(1)\cap F_{\alpha}=\set{1}$ for such $\delta$ and this isolated point can be covered by a single ball of any radius.
\end{proof}

In fact, the above argument demonstrates that any set with an isolated point must have lower Assouad dimension equal to $0$. This could be viewed as an undesirable property for a dimension to have, as adding an isolated point to a set $F$ with $\dim_{\rm LA}F > 0$ has the effect of \emph{reducing} the lower Assouad dimension to zero.

Equi-homogeneous sets, which we define in Section \ref{section - equi-homogeneity}, are those for which the quantity $N\left(B_{\delta}\left(x\right)\cap F,\rho\right)$ scales identically at every point $x\in F$. We will see that the equi-homogeneity property essentially removes the local dependence in the definitions of the Assouad and lower Assouad dimensions. Before we introduce equi-homogeneity we finish this section by recalling some familiar notions of regularity.

\subsection{Regularity of sets} 
Dimensional equality is a common notion of regularity of sets, which is enjoyed by all smooth manifolds. Indeed, Mandelbrot \cite{Mandelbrot75} first defined ``fractal'' sets as those with unequal topological and Hausdorff dimension, although this definition has fallen out of favour in recent years. Nevertheless, equality of dimensions is sufficient for some good properties of sets. For example if the Hausdorff and upper box-counting dimensions of $F\subset \R^{n}$ coincide, then the product set inequality
\[
\dim_{\rm H}\left(F\times E\right) \geq \dim_{\rm H} F + \dim_{\rm H} E
\]
is actually an equality for arbitrary $E\subset \R^{m}$ (see Corollary 7.4 of Falconer \cite{BkFalconer14}). Equality in the Assouad and lower Assoaud dimensions is particularly powerful: it follows from the inequalities \eqref{dimension inequalities 1} and \eqref{dimension inequalities 2} that if the Assouad and lower Assouad dimensions coincide then all of these dimensions agree.

We now recall the definition of Ahlfors-David regularity.
\begin{definition}\label{definition - Ahlfors-David}
A bounded set $F\subset X$ is \emph{Ahlfors-David $s$-regular} if there exists a constant $C>0$ such that
\begin{align}\label{Ahlfors-David inequality}
C^{-1} \delta^{s} \leq \mathcal{H}^{s}\left(B_{\delta}\left(x\right)\cap F\right) \leq C \delta^{s}
\end{align}
for all $x\in F$ and all $0<\delta <\diameter F$, where $\mathcal{H}^{s}$ is the usual $s$-dimensional Hausdorff measure.
\end{definition}
If $F$ is Ahlfors-David $s$-regular then by taking $\delta=\diameter F$ it immediately follows that $0<\mathcal{H}^{s}\left(F\right)<\infty$, which is precisely that $F$ is an `$s$-set' (see Falconer \cite{BkFalconer14} pp.48) so in particular $\dim_{\rm H}F=s$.

Remarkably, if \eqref{Ahlfors-David inequality} holds with $\mathcal{H}^{s}$ replaced by \emph{any} Borel measure then it follows that $F$ is Ahflors-David $s$-regular (see Chapter 8 of Heinonen \cite{Heinonen}).

Ahlfors-David regularity is a strong notion of regularity in the sense that it guarantees the equality of all the dimensions in \eqref{dimension inequalities 1}. We prove this in part of Theorem \ref{theorem - ahlfors implies equih} in the next section.

\section{Equi-homogeneity}\label{section - equi-homogeneity}

From Definitions \ref{definition - Assouad dimension} and \ref{definition - lower Assouad dimension} we see that the Assouad and lower Assouad dimensions respectively capture the maximum and minimum cardinalities of local covers. The distinct values for the dimensions in the example of Proposition \ref{proposition - different scaling} reflect the different cardinalities of covers at the `high detail' limit point $0\in F$ and the `low detail' isolated point $1\in F$.

We introduce the equi-homogeneity property to examine sets where there is no such `local dependence' on the cardinalities of local covers. Roughly this means that at each fixed length-scale an equi-homogeneous set exhibit identical dimensional detail near every point. 

\begin{definition}\label{equihom}
  We say that a set $F\subset X$ is \emph{equi-homogeneous} if for all $\delta_{0}>0$ there exist constants $M\geq 1$ and $c_{1},c_{2}>0$ such that
  \begin{align}\label{equihom inequality}
    \sup_{x\in F} \cover(B_{\delta}(x)\cap F,\rho)&\leq M
    \inf_{x\in F} \cover(B_{c_{1}\delta}(x)\cap F,c_{2}\rho)
  \end{align}
  for all $\delta,\rho$ with $0<\rho<\delta\leq \delta_{0}$.
\end{definition}

Note that as $\cover(B_{\delta}(x)\cap F,\rho)$ increases with $\delta$ and decreases with $\rho$, by replacing the $c_{i}$ with $1$ if necessary we can assume without loss of generality that $c_{2}\leq 1 \leq c_{1}$ in \eqref{equihom inequality}.

\subsection{Equivalent definitions}
As with the definitions of the Assouad dimensions, for a large class of sets it is sufficient that \eqref{equihom inequality} holds only for \textit{some} $\delta_{0}$.

\begin{lemma}\label{lemma - equiextension}
  If $F\subset X$ is totally bounded or $X$ is a locally totally bounded space then $F$ is equi-homogeneous if and only if there exist constants $M\geq 1$ and $c_{1},c_{2},\delta_{1}>0$ such that
  \begin{align}
    \sup_{x\in F} \cover(B_{\delta}(x)\cap
      F,\rho)&\leq M \inf_{x\in F}
    \cover(B_{c_{1}\delta}(x)\cap F,c_{2}\rho)\label{equihom equivalent}
  \end{align}
  for all $\rho,\delta$ satisfying $0<\rho<\delta\leq\delta_{1}$.
\end{lemma}

\begin{proof}
It is sufficient to prove that for each $\delta_{0}>0$ the inequality \eqref{equihom equivalent} can be extended to hold for all $\delta,\rho$ with $0<\rho<\delta\leq \delta_{0}$ up to a change in constant $M$. Let $\delta_{0}>0$ be arbitrary. If $\delta_{0}\leq \delta_{1}$ then there is nothing to prove, so we assume that $\delta_{0}>\delta_{1}$.
Suppose that $\delta,\rho$ lie in the range $0<\rho<\delta_{1}<\delta\leq \delta_{0}$ and let $x\in F$ be arbitrary. From Lemma \ref{lemma - refinement} with $r=\delta_{1}$ we obtain
\begin{align}
    \cover(B_{\delta}(x)\cap F,\rho)
	&\leq \cover(B_{\delta}(x)\cap F,\delta_{1})
		\sup_{x\in F} \cover(B_{\delta_{1}}
			(x)\cap F,\rho)\notag\\
    &\leq \cover(B_{\delta}(x)\cap F,\delta_{1})M \inf_{x\in F} \cover \left(B_{c_{1}\delta_{1}}\left(x\right)\cap F,c_{2}\rho\right)\notag\\
    &\leq \cover(B_{\delta}(x)\cap F,\delta_{1})M \inf_{x\in F} \cover \left(B_{c_{1}\delta}\left(x\right)\cap F,c_{2}\rho\right)\label{extension bound small rho}
  \end{align}
which follows from \eqref{equihom equivalent} and the fact that $\delta>\delta_{1}$.

Taking the first case assume that $X$ is a locally totally bounded space. It follows from \eqref{extension bound small rho} that for
  $0<\rho<\delta_{1}<\delta\leq \delta_{0}$
  \[
    \cover(B_{\delta}(x)\cap F,\rho)\leq
    \cover(B_{\delta_{0}}(0),\delta_{1}) M \inf_{x\in F} \cover \left(B_{c_{1}\delta}\left(x\right)\cap F,c_{2}\rho\right),
  \]
  and trivially for
      $\delta_{1}\leq\rho<\delta\leq\delta_{0}$ that
  \[
    \cover(B_{\delta}(x)\cap F,\rho)\leq
    \cover(B_{\delta}(x),\rho)\leq
    \cover(B_{\delta_{0}}(x),\delta_{1})\leq
    \cover(B_{\delta_{0}}(0),\delta_{1})
    M \inf_{x\in F} \cover \left(B_{c_{1}\delta}\left(x\right)\cap F,c_{2}\rho\right)
  \]
  as $M \inf_{x\in F} \cover \left(B_{c_{1}\delta}\left(x\right)\cap F,c_{2}\rho\right)\geq 1$.  Consequently, with $M_{\delta_{0}}=\cover(B_{\delta_{0}}(0),\delta_{1})M$ we obtain
  \begin{align*}
    \sup_{x\in F}\cover(B_{\delta}(x)\cap
      F,\rho)&\leq M_{\delta_{0}} \inf_{x\in F}\cover(B_{c_{1}\delta}(x)\cap
      F,c_{2}\rho)
    \qquad &\forall\
    \delta,\rho\quad\text{with}\quad0<\rho<\delta\le\delta_0,
  \end{align*}
  so the constant $M_{\delta_{0}}$ is sufficient to extend \eqref{equihom equivalent} to all $0<\rho<\delta\leq \delta_{0}$.

  Taking the second case assume that $F \subset X$ is totally bounded. It follows from \eqref{extension bound small rho} that for
  $0<\rho<\delta_{1}<\delta\leq \delta_{0}$
  \begin{align*}
    \cover(B_{\delta}(x)\cap F,\rho)&\leq
    \cover(F,\delta_{1}) M \inf_{x\in F} \cover \left(B_{c_{1}\delta}\left(x\right)\cap F,c_{2}\rho\right),
    \intertext{and again for $\delta_{1}\leq \rho<\delta\leq
      \delta_{0}$ that} \cover(B_{\delta}(x)\cap
      F,\rho)&\leq \cover(F,\delta_{1})\leq \cover\left(F,\delta_{1}\right) M \inf_{x\in F} \cover \left(B_{c_{1}\delta}\left(x\right)\cap F,c_{2}\rho\right).
  \end{align*}
  Consequently, the constant $M^{\prime}=\cover(F,\delta_{1})M$ is sufficient to extend \eqref{extension bound small rho} to all $0<\rho<\delta\leq \delta_{0}$.
  
The converse implication follows immediately from the definition of equi-homogeneity.
\end{proof}

We remark that for totally bounded sets $F\subset X$ the constant $M^{\prime}$ does not depend upon $\delta_{0}$, so the inequality \eqref{equihom equivalent} can be extended to hold for all $\delta,\rho$ with $0<\rho<\delta$.
If $F$ is totally bounded then we can find $M\geq 1$ such that \eqref{equihom inequality} holds for all $\rho,\delta$ with $0<\rho<\delta$.

In normed spaces that are locally totally bounded (such as Euclidean space) there is an even more elementary formulation that does not require the constants $c_{1},c_{2}$.
\begin{lemma}\label{lemma - euclideanequiextension}
  Let $X$ be a normed space that is locally totally bounded. A set
  $F\subset X$ is equi-homogeneous if and only if there exist
  constants $M\geq 1$, $\delta_{1}\geq 1$ such that
  \begin{align}
    \sup_{x\in F} \cover(B_{\delta}(x)\cap F,\rho)&\leq M
    \inf_{x\in F} \cover(B_{\delta}(x)\cap F,\rho)
  \end{align}
  for all $\rho,\delta$ with $0<\rho<\delta\leq \delta_{1}$.
\end{lemma}
\begin{proof}
  The `if' direction follows immediately from Lemma \ref{lemma - equiextension}. To prove the converse fix $\delta_{0}>0$ and let $M\geq 1$ and $c_{1},c_{2}>0$ with $c_{2}\leq 1 \leq c_{1}$ be such that
  \begin{align*}
    \sup_{x\in F} \cover(B_{\delta}(x)\cap F,\rho)
    \leq M \inf_{x\in F} \cover(B_{c_{1}\delta}(x)\cap
      F,c_{2}\rho)
  \end{align*}
  for all $0<\rho<\delta\leq \delta_{0}$.

  First, observe that replacing $\delta$ by $\delta/c_1$ we can assume that
  \begin{equation}\label{movec1}
    \sup_{x\in F} \cover(B_{\delta/c_1}(x)\cap F,\rho)\leq M
    \inf_{x\in F} \cover(B_{\delta}(x)\cap F,c_{2}\rho)
  \end{equation}
  for all $\delta,\rho$ with $0<\rho<\delta/c_1$, $\delta\le c_1\delta_0$. Note that if $\rho\geq \delta/c_1$ then the above inequality holds trivially, since the left-hand side is $1$ and the right-hand side is at least $M\ge 1$; so in fact \eqref{movec1} holds for all $0<\rho<\delta\leq \delta_{1}:= c_1\delta_0$.

  Now, it follows from \eqref{refinement} with $r=\delta/c_{1}$ that
  \begin{align}
    \cover(B_{\delta}(x)\cap F,\rho)&\leq
    \cover(B_{\delta}(x),\delta/c_{1}) \sup_{x\in F}
    \cover(B_{\delta/c_{1}}(x)\cap F,\rho) \notag
    \intertext{for all $x\in F$, so setting $N_{1}\define
      \cover(B_{\delta}(x),\delta/c_{1}) =
      \cover(B_{1}(0),1/c_{1})$, which follows as $X$
      is a normed space, we obtain} \sup_{x\in F}
    \cover(B_{\delta}(x)\cap F,\rho) &\leq N_{1}
    \sup_{x\in F} \cover(B_{\delta/c_{1}}(x)\cap
      F,\rho).\label{equiextension 1}
  \end{align}
  It also follows from \eqref{refinement} that for any $r>0$
  \begin{align}
    \cover(B_{\delta}(x)\cap F,c_{2}\rho)
	&\leq \cover(B_{\delta}(x)\cap F,r)
	\sup_{x\in F}\cover(B_{r}(x)\cap F,c_{2}\rho)\notag\\
    &\leq \cover(B_{\delta}(x)\cap F,r)
		\sup_{x\in F}\cover(B_{r}(x),c_{2}\rho)\notag\\
    &=\cover(B_{\delta}(x)\cap
      F,r)\cover(B_{r}(0),c_{2}\rho)\notag
    \intertext{so taking $r=\rho$, setting
      $N_{2}=\cover(B_{\rho}(0),c_{2}\rho)
		=\cover(B_{1}(0),c_{2})$,
      which again follows as $X$ is a normed space, and taking the
      infimum over $x\in F$ we obtain} \inf_{x\in
      F}\cover(B_{\delta}(x)\cap F,c_{2}\rho) &\leq
    \inf_{x\in F} \cover(B_{\delta}(x)\cap F,\rho)
    N_{2}. \label{equiextension 2}
  \end{align}
  It follows from \eqref{movec1}, \eqref{equiextension 1} and
  \eqref{equiextension 2} that for all $\rho,\delta$ with
  $0<\rho<\delta\leq \delta_{1}$
  \begin{align*}
    \sup_{x\in F}\cover(B_{\delta}(x)\cap F,\rho)
    \leq M\frac{N_{1}}{N_{2}}\inf_{x\in F}
    \cover(B_{\delta}(x)\cap F,\rho)
  \end{align*}
  so we conclude from Lemma \ref{lemma - euclideanequiextension} that
  $F$ is equi-homogeneous.
\end{proof}

For reasonable choices of product metric the product of two equi-homogeneous sets is also equi-homogeneous  (see Olson, Robinson \& Sharples \cite{ORS}).

\subsection{Equi-homogeneity and the Assouad dimensions}\label{section - equi-homogeneity and assouad}
Equi-homogeneous sets have identical dimensional detail near every point.
The Assouad and lower Assouad dimensions of an equi-homogeneous set can be found by examining the cardinality of the local cover at an arbitrary point.

\begin{lemma}\label{lemma - Assouad def for equihom}
If $F\subset X$ is equi-homogeneous then for any $x\in F$
\begin{enumerate}
\item $\dim_{\rm A} F$ is the infimum over all $s\in \R^{+}$ such that for all $\delta_{0}>0$ there exists a constant $C>0$ for which \label{Assouad def for equihom}
\begin{align}\label{Assouad scale equihom}
\cover\left(B_{\delta}\left(x\right)\cap F,\rho\right) \leq C\left(\delta/\rho\right)^{s}\quad \forall\ \delta,\rho \quad \text{with} \quad 0<\rho<\delta\leq\delta_{0}.
\end{align}
and,
\item $\dim_{\rm LA} F$ is the supremum over all $s\in \R^{+}$ such that for all $\delta_{0}>0$ there exists a constant $C>0$ for which
\[
\cover\left(B_{\delta}\left(x\right)\cap F,\rho\right) \geq C\left(\delta/\rho\right)^{s}\quad \forall\ \delta,\rho \quad \text{with} \quad 0<\rho<\delta\leq\delta_{0}.
\]\label{lower Assouad def for equihom}
\end{enumerate}
\end{lemma}
\begin{proof}
Let $s^{*}$ be the infimum in \ref{Assouad def for equihom}.
Clearly \eqref{Assouad scale} implies \eqref{Assouad scale equihom} for all $x\in F$, so $\dim_{\rm A} F \geq s^{*}$.
Next, let $\delta_{0}>0$ be arbitrary. As $F$ is equi-homogeneous there exist constants $c_{1},c_{2},M>0$ with $c_{2}\leq 1 \leq c_{1}$ such that for all $x\in F$
\begin{align}\label{equihom at x}
\sup_{x\in F} \cover \left(B_{\delta}\left(x\right)\cap F,\rho\right) \leq M \cover \left(B_{c_{1}\delta}\left(x\right)\cap F,c_{2}\rho\right) \quad \forall\ \delta,\rho \quad \text{with} \quad 0<\rho<\delta\leq\delta_{0}
\end{align}

If $s>s^{*}$ then there exists a constant $C>0$ such that
\begin{align*}
\cover\left(B_{\delta}\left(x\right)\cap F,\rho\right) &\leq C\left(\delta/\rho\right)^{s}\quad \forall\ \delta,\rho \quad \text{with} \quad 0<\rho<\delta\leq c_{1}\delta_{0}.
\intertext{As $0<\rho<\delta\leq\delta_{0}$ implies $0<c_{2}\rho<c_{1}\delta\leq c_{1}\delta_{0}$ it follows that}
M
\cover \left(B_{c_{1}\delta}\left(x\right)\cap F,c_{2}\rho\right) &\leq MC\left(c_{1}\delta/c_{2}\rho\right)^{s} \quad \forall\ \delta,\rho \quad \text{with} \quad 0<\rho<\delta\leq\delta_{0}
\intertext{hence from \eqref{equihom at x}}
\sup_{x\in F} \cover \left(B_{\delta}\left(x\right)\cap F,\rho\right) &\leq MC\left(c_{1}/c_{2}\right)^{s}\left(\delta/\rho\right)^{s} \quad \forall\ \delta,\rho \quad \text{with} \quad 0<\rho<\delta\leq\delta_{0}.
\end{align*}
It follows that $s\geq \dim_{\rm A} F$, but as $s>s^{*}$ was arbitrary we conclude that $s^{*}\geq \dim_{\rm A}F$, hence $s^{*}=\dim_{\rm A}F$. Part \ref{lower Assouad def for equihom} follows similarly.
\end{proof}

In fact, for equi-homogeneous sets whose box-counting dimensions are equal and attained (i.e. \eqref{box-counting growth bounds} holds with $\varepsilon=0$) then can find the Assouad and lower Assouad dimensions in terms of the more elementary box-counting dimensions, which do not account for variance in local detail.

\begin{theorem}\label{theorem - equi-homogeneous Assouad equal to box}
  If a totally bounded set $F\subset X$ is equi-homogeneous, $F$
  attains both its upper and lower box-counting dimensions, and
  $\dim_{\rm LB}F=\dim_{\rm B}F$, then
\begin{align}\label{attainment theorem equality}
\dim_{\rm A}F=\dim_{\rm B}F=\dim_{\rm LB}F=\dim_{\rm LA}F
\end{align}
\end{theorem}
\begin{proof}
See Olson, Robinson \& Sharples \cite{ORS} for the first equality of \eqref{attainment theorem equality}. A simple adaptation of their argument yields the last equality.
\end{proof}

Consequently, equi-homogeneous sets with equal, attained box-counting dimensions satisfy the strong dimensional equality $\dim_{\rm LA}F=\dim_{\rm A} F$.

In general, self-similar sets can have distinct Assouad and lower Assouad dimensions (see Fraser \cite{Fraser2014}), although these dimensions coincide for self-similar sets that satisfy the Moran open-set condition (see Corollary 2.11 of Fraser \cite{Fraser2013}). In Theorem \ref{theorem - homogeneous case} we show that sets in this large class are also equi-homogeneous, however we will now show that equi-homogeneity is not equivalent to the condition $\dim_{\rm A}F=\dim_{\rm LA}F$.

The Assouad and lower Assouad dimensions encode the scaling of the two distinct quantities $\sup_{x\in F} \cover(B_\delta(x)\cap F,\rho)$, and $\inf_{x\in F} \cover(B_\delta(x)\cap F,\rho)$, which can scale very differently as we have seen in the example of Proposition \ref{proposition - different scaling}.
However, in light of Lemma \ref{lemma - Assouad def for equihom}, for equi-homogeneous sets the Assouad and lower Assouad dimensions encode the scaling of the single quantity $\cover\left(B_{\delta}\left(x\right)\cap F,\rho\right)$ in a similar way to the box-counting dimension scaling law \eqref{box-counting growth bounds}.

Even though the minimal and maximal local covers scale identically for an equi-homogeneous set, the Assouad and the lower Assouad dimensions can be distinct. Rather than being the consequence of the different scaling of the minimal and maximal local cover (as in the example of Proposition \ref{proposition - different scaling}), any difference in these dimensions is due to an oscillation in the growth of the covers $\cover\left(B_{\delta}\left(x\right)\cap F,\rho\right)$. If this oscillation is sufficiently large then the upper and lower bounds on the growth of the covers (provided respectively by the Assouad and lower Assouad dimensions) are distinct. This is analogous to how distinct box-counting dimensions are due to a large oscillation in the growth of $\cover\left(F,\delta\right)$ (see Robinson \& Sharples \cite{RobinsonSharples13RAEX}).

We now give an example of the oscillating growth of local covers, and the resulting difference in the Assouad dimensions. As shown in the introduction, the generalised Cantor sets that were constructed in Section 3 of Olson, Robinson \& Sharples \cite{ORS} may also be obtained as the pullback attractors of the non-autonomous iterated function systems given by \eqref{gencant}.
These systems satisfy the hypothesis of Theorem \ref{theorem - homogeneous case} (with Moran open sets $U^k=(0,1)$ for all $k$) and are therefore equi-homogeneous (we give a direct proof that generalised Cantor sets are equi-homogeneous in Olson, Robinson \& Sharples \cite{ORS}).
At the same time, these sets can have distinct Assouad dimensions. In fact, even the Assouad and the upper box-counting dimension can differ for these sets. We give a simple example here to demonstrate this.

\begin{proposition}\label{proposition - equihom distinct dimension}
Let $F$ be the generalised Cantor set with contraction ratios
\begin{align*}
	c_k=
	\begin{cases}
	1/3 & \text{for}\quad k\in[2^{2(n-1)},2^{2n-1})\\
		1/9 & \text{for}\quad k\in [2^{2n-1},2^{2n}).
	\end{cases}
\end{align*}
Then
\[
\dim_{\rm B}F= \frac{3}{5} \frac{\log 2}{\log 3} < \frac{\log 2}{\log 3} \leq\dim_{\rm A}F.
\]
\end{proposition}

\begin{proof}
Let $\pi_k=c_1\cdots c_k$.
By results from Olson, Robinson \& Sharples \cite{ORS}, see also Hua, Rao, Wen \& Wu~\cite{Hua2000}, the upper box-counting dimension of $F$ is given by
\[
	\dim_{\rm B}F=\limsup_{k\to\infty} s_k \wwords{where}
 s_k=\frac{k\log2}{\log(1/\pi_k)}.
\]

Since $s_k$ is a non-increasing function on $[2^{2(n-1)},2^{2n-1})$ and an increasing function on $[2^{2n-1},2^{2n})$, it follows that we may take the limit supremum along the subsequence $k=4^n$.
The calculation
\begin{align*}
	\log (1/\pi_{4^n})
	&=\sum_{j=1}^n \Big(2^{2(j-1)}\log 3+2^{2j-1}\log 9\Big)\\
	&=\frac{5}{4}\sum_{j=1}^n 4^{j}\log 3
	=\frac{5}{3} (4^{n}-1)\log 3
\end{align*}
therefore implies
\[
	\dim_{\rm B}F=\lim_{n\to\infty} \frac{4^n\log 2}{(5/3)(4^n-1)\log 3}
		=\frac{3}{5} \frac{\log 2}{\log 3}.
\]

On the other hand, $\dim_{\rm A}F\ge {\log 2/\log 3}$.
Suppose, for contradiction, that $\dim_{\rm A}F<s<\log 2/\log 3$.
Then, there would exist a constant $C$ such that
\[
	\cover(B_\delta(x)\cap F,\rho/2)\le C(\delta/\rho)^s
\words{for all} 0<\rho<\delta.
\]
Choose $\delta_n=\pi_{2^{2(n-1)}}$ and $\rho_n=\pi_{2^{2n-1}}$.
Then
\[
\delta_n/\rho_n=3^{2^{2(n-1)}}
\wwords{and}
	\cover(B_{\delta_n}(x)\cap F,\rho_n/2) \ge 2^{2^{2(n-1)}}
\]
would imply that
\[
1 \le C 2^{-2^{2(n-1)}}
	(3^{2^{2(n-1)}})^s
	= C (3^{2^{2(n-1)}})^{s-\log2/\log3} \to 0
\words{as} n\to\infty
\]
which is a contradiction.
\end{proof}

We note that the above set $F$ satisfies the Moran structure conditions of Wen \cite{Wen2001} with $c_*=\inf\{\, c_k : k\in\N\,\}=1/9>0$.
Therefore an additional assumption, such as~\eqref{hippo}, is needed in order to complete the proof of Li \cite{Li2013} and conclude that the Assouad and upper box-counting dimensions are equal.

We now give a simple example of a set that satisfies $\dim_{\rm A}F=\dim_{\rm LA}F$ but that is not
equi-homogeneous. Taken together these two examples demonstrate that the notion of equi-homogeneity is entirely distinct from the coincidence of these two dimensions.

\begin{proposition}\label{proposition - equal dimension not equih}
The set $F=\{0,1\}\cup \{\, 2^{-n} : n\in \mathbb{N}\,\}$ satisfies $\dim_{\rm A}F=\dim_{\rm LA}F=0$ but $F$ is not equi-homogeneous.
\end{proposition}
We remark that the equality $\dim_{\rm A}F=0$ is stated as Fact 4.3 in~Olson \cite{EJO} without proof. We also remark that the logarithmic terms that occur in the course of the proof can also be used to show that $F$ does not
`attain' its box-counting dimension (i.e. that \eqref{box-counting growth bounds} does not hold with $\varepsilon=0$.).

\begin{proof}
Let $\delta=1/4$, then
\[
B_\delta(1)\cap F= \{1\} \quad\hbox{implies that}\quad \inf_{x\in F}
\cover(B_\delta(x)\cap F,\rho) = 1
\]
for every $\rho>0$.  On the other hand, for $0<\rho<1/4$, let $K$ be chosen so that
\[
2^{-K-1}\le \rho< 2^{-K}.
\]
Then $K\ge2$ and
\[
B_\delta(0)\cap F \supseteq \{\, 2^{-n} : n=2,\ldots, K\,\}.
\]
Moreover $2^{-n+1}-2^{-n}=2^{-n}\ge 2^{-K}>\rho$ for $n\le K$ implies that at least one set of diameter $\rho$ is required to cover each of the $K-1$ points above. Therefore
\[
\sup_{x\in F} \cover(B_\delta(x)\cap F,\rho) \ge K-1 \ge \frac{\log
  (1/\rho)}{\log 2} -2.
\]
This shows there is no value for $M$ independent of $\rho$ that could appear in Definition \ref{equihom} for this set, and so $F$ is not equi-homogeneous.

Clearly $\dim_{\rm LA}F=0$. To show that $\dim_{\rm A}F=0$ let $x\in [0,1]$ and $0<\rho<\delta<1/4$. Define
\[
	G=\{\, 2^{-n} : \max(0,x-\delta) < 2^{-n}\le \rho\,\}
\]
and
\[
	H=\{\, 2^{-n} : \max(\rho,x-\delta) < 2^{-n}< \min(x+\delta,1)\,\}.
\]
Then $B_\delta(x)\cap F\subseteq \{0,1\}\cup G\cup H.$ Now depending on $\rho$, $x$, and $\delta$ it may happen that either or both of the sets $H$ and $G$ are empty. As covering an empty set is trivial, we need only consider the cases when these sets are non-empty.

If $G\ne\emptyset$ then $x-\delta<\rho$, and it follows that
\begin{equation}\label{gest}
  \cover(G,\rho)\le \frac{\rho-\max(0,x-\delta)}{\rho} +1\le 2.
\end{equation}
Similarly if $H\ne\emptyset$ then
\[
\cover(H,\rho)\le \frac{1}{\log 2}
\log\bigg\{\frac{\min(x+\delta,1)}{\max(\rho,x-\delta)}\bigg\} +1.
\]
If $x+\delta\ge 1$ then $x-\delta\ge 1-2\delta\ge 1/2.$ Thus $\cover(H,\rho)\le 2$.  If $x-\delta\le \rho$ then $x+\delta\le \rho+2\delta<3\delta<1$.  Thus $\cover(H,\rho)\le (\log 2)^{-1}\log(3\delta/\rho)+1$. Otherwise, $\rho+\delta<x<1-\delta$.
On this interval $ x\mapsto \log\big\{(x+\delta)/(x-\delta)\big\} $ is a decreasing function. Therefore, in general,
\begin{equation}\label{hest}
  \cover(H,\rho)\le 2\log(\delta/\rho)+3.
\end{equation}

Combining \eqref{gest} with \eqref{hest} we obtain
\[
\cover(B_\delta(x)\cap F,\rho)\le 2\log(\delta/\rho)+7
\]
Since for every $s>0$ there exists $C>0$ such that
\[
2\log(\delta/\rho)+7\le C(\delta/\rho)^s \qquad\hbox{for every}\qquad
0<\rho<\delta<1/4.
\]
It follows from the remarks after Definition \ref{definition - lower Assouad dimension} that $\dim_{\rm A}F=0$.
\end{proof}

We finish this section by showing that Ahlfors-David regularity implies equi-homogeneity. Along the way we will show that Ahlfors-David regularity implies equality of the Assouad and lower Assouad dimensions.

\begin{theorem}\label{theorem - ahlfors implies equih}
Let $F\subset X$ be a bounded set such that either $F$ is totally bounded or $X$ is a locally totally bounded space. If $F \subset X$ is Ahlfors-David $s$-regular then $F$ is equi-homogeneous and $\dim_{\rm LA}F=\dim_{\rm A}F=s$.
\end{theorem}
\begin{proof}
As $F$ is Ahlfors-David $s$-regular there exists a constant $C>0$ such that
\begin{align}\label{F Ahlfors-David}
C^{-1} \delta^{s} \leq \mathcal{H}^{s}\left(B_{\delta}\left(x\right)\cap F\right) &\leq C \delta^{s} & 0<\delta<\diameter F.
\end{align}
Let $\delta_{0}=\diameter\left(F\right) / 2 $, fix $x,y\in F$ and $\delta,\rho$ with $0<\rho<\delta<\delta_{0}$. Observe that $N:=\cover\left(B_{\delta}\left(x\right)\cap F,\rho\right)<\infty$ as either $F$ is totally bounded or $X$ is locally totally bounded so certainly $B_{\delta}\left(x\right)\cap F$ is totally bounded. Let $x_{i} \in B_{\delta}\left(x\right)\cap F$ for $i=1,\ldots, N$ be the centres of $\rho$-balls that cover $B_{\delta}\left(x\right)\cap F$. Clearly $\bigcup_{i=1}^{N} B_{\rho}\left(x_{i}\right) \supset B_{\delta}\left(x\right) \cap F$ so
\begin{align}
\mathcal{H}^{s}\left(\bigcup_{i=1}^{N} B_{\rho}\left(x_{i}\right)\cap F \right) &\geq \mathcal{H}^{s}\left(B_{\delta}\left(x\right)\cap F\right) \geq C^{-1}\delta^{s}\label{Ahlfors first inequality}
\intertext{from \eqref{F Ahlfors-David}, while}
\mathcal{H}^{s}\left(\bigcup_{i=1}^{N} B_{\rho}\left(x_{i}\right)\cap F \right) &\leq \sum_{i=1}^{N} \mathcal{H}^{s}\left(B_{\rho}\left(x_{i}\right)\cap F\right) \leq \sum_{i=1}^{N} C\rho^{s} = NC\rho^{s}. \label{Ahlfors second inequality}
\end{align}
Combining \eqref{Ahlfors first inequality} and \eqref{Ahlfors second inequality} we obtain
\begin{equation}\label{Ahlfors theorem lower bound}
\cover\left(B_{\delta}\left(x\right)\cap F,\rho\right) = N \geq C^{-2}\left(\delta/\rho\right)^{s}
\end{equation}
for all $\rho,\delta$ with $0<\rho<\delta\leq\delta_{0}$ and all $x\in F$. It follows that $\dim_{\rm LA} F \geq s$ (see the remarks after Definition \ref{definition - lower Assouad dimension}).

Next, observe that $P:= \pack \left(B_{\delta}\left(y\right)\cap F,\rho\right)<\infty$ and let $y_{i}\in B_{\rho}\left(y\right)\cap F$ for $i=1,\ldots, P$ be the centres of disjoint $\rho$-balls. As $B_{\rho}\left(y_{i}\right)\subset B_{2\delta}\left(y\right)$ for each $i$ it follows that $\bigcup_{i=1}^{P} B_{\rho}\left(y_{i}\right)\cap F \subset B_{2\delta}\left(y\right) \cap F$ so
\begin{align}
\mathcal{H}^{s}\left(\bigcup_{i=1}^{P} B_{\rho}\left(y_{i}\right)\cap F\right) \leq \mathcal{H}^{s}\left(B_{2\delta}\left(y\right)\cap F\right) \leq C2^{s}\delta^{s} \label{Ahlfors third inequality}
\intertext{from \eqref{F Ahlfors-David}. Further, as the $B_{\rho}\left(y_{i}\right)$ are disjoint it follows that}
\mathcal{H}^{s}\left(\bigcup_{i=1}^{P} B_{\rho}\left(y_{i}\right)\cap F\right) = \sum_{i=1}^{P} \mathcal{H}^{s}\left(B_{\rho}\left(y_{i}\right)\cap F\right) \geq \sum_{i=1}^{P} C^{-1}\rho^{s} = PC^{-1}\rho^{s}\label{Ahlfors fourth inequality}
\end{align}
Combining \eqref{Ahlfors third inequality} and \eqref{Ahlfors fourth inequality} with the geometric inequalities \eqref{geometric inequalities} we obtain
\begin{equation}\label{Ahlfors theorem upper bound}
\cover \left(B_{\delta}\left(y\right)\cap F,\rho\right) \leq \pack \left(B_{\delta}\left(y\right)\cap F,\rho\right) = P \leq C^{2}2^{s} \left(\delta/\rho\right)^{s}
\end{equation}
for all $\rho,\delta$ with $0<\rho<\delta\leq \delta_{0}$ and all $y\in F$. It follows that $\dim_{\rm A}F \leq s$ (see the remarks after Definition \ref{definition - lower Assouad dimension}). As we now have $s \leq \dim_{\rm LA}F\leq \dim_{\rm A}F \leq s$ we must have equality throughout. Finally, it follows from \eqref{Ahlfors theorem lower bound} and \eqref{Ahlfors theorem upper bound} that
\begin{align*}
\cover \left(B_{\delta}\left(y\right)\cap F,\rho\right) \leq C^{2}2^{s}\left(\delta/\rho\right)^{s} \leq C^{4}2^{s} \cover \left(B_{\delta}\left(x\right)\cap F,\rho\right)
\end{align*}
for all $\rho,\delta$ satisfying $0<\rho<\delta\leq \delta_{1}$ and arbitrary $x,y\in F$. Taking limits we conclude
\begin{align*}
\sup_{x\in F}\cover \left(B_{\delta}\left(x\right)\cap F,\rho\right) \leq C^{4}2^{s} \inf_{x\in F} \cover \left(B_{\delta}\left(x\right)\cap F,\rho\right)
\end{align*}
for all $\rho,\delta$ satisfying $0<\rho<\delta\leq \delta_{1}$, so $F$ is equi-homogeneous by Lemma \ref{lemma - equiextension}.
\end{proof}

\section{Equi-homogeneity and dynamical systems}

We now demonstrate that the notion of equi-homogeneity is not overly restrictive: it is enjoyed by all self-similar sets that satisfy the Moran open-set condition, generalised Cantor sets, and the pullback attractors for a certain class of non-autonomous iterated function systems that satisfy the (generalised) Moran open-set condition.

\subsection{Autonomous systems}\label{IFS1}

We will begin with self-similar sets, which are a much studied and canonical class of fractal sets. Classical (`autonomous') iterated function systems are precisely the non-autonomous iterated function systems that satisfy $\I_k=\I$ for each $k\in\N$. If we define the set function $\T$ and its iterates as
\[
	\T(B)=\bigcup_{i\in\I} f_i(B)
\wwords{and}
	\T^{n+1}(B)=\T\circ \T^{n}(B)
\]
where $\T^{0}(B)=B$ then it is well known (see Falconer \cite{BkFalconer14} for example,) that if the $f_{i}$ are contractions then
there exists a unique non-empty compact set $F$ such that
\begin{equation}
  F={\mathcal T}(F) \label{attract}
\end{equation}
called the attractor of the iterated function system. Further, the attractor satisfies $\rho_H(\T^l(B),F)\to 0$ as $l\to\infty$ for any bounded set $B\subset\R^d$. In this case, defining $F^k=F$ for all $k\in\N_0$ and noting that $\T^l=\S^{0,l}$ yields a collection of compact sets satisfying Definition \ref{definition - pullback attractor} so the collection $\set{F^{k}}$ is a pullback attractor.
Conversely, as pullback attractors are unique (which we prove in Theorem \ref{unique}) it follows that for classical iterated function systems the pullback attractor can be identified with the invariant set.

A set $F$ is said to be self-similar if it is the attractor of a classical system of contracting similarities. Reasonable self-similar sets result when we impose some separation properties on the iterated function system, see Falconer \cite{BkFalconer14} pp. 139 or Hutchinson \cite{Hutch}, for example. The simplest such property is the Moran open-set condition (see Moran \cite{Moran}): there exists an open set $U$ such that $F\subset\overline U$, $f_i(U)\subseteq U$, and
\[
f_i(U)\cap f_j(U)=\emptyset\qquad\mbox{when}\qquad i\neq j.
\]

If a classical (`autonomous') iterated function system satisfies the Moran open-set condition, then it satisfies the generalised Moran open-set condition of Definition \ref{definition - generalised MOSC} with $U^k=U$, $\I_k=\I$ and $F^k=F$ for every $k\in\N_0$. 
Thus, this definition of the generalised Moran open-set condition reduces to the classical one when the functions are the same at every step of the iteration, and hence there is no ambiguity in referring to the generalised Definition \ref{definition - generalised MOSC} simply as the `Moran open-set condition'.

These self-similar sets are equi-homogeneous, which is the content of Theorem \ref{theorem - self-similar} from the Introduction.

\mainzero*
\begin{proof}
Self-similar sets that satisfy the open-set condition are Ahlfors-David regular (Theorem 1(i), Section 5.3 of Hutchinson \cite{Hutch}) so by Theorem \ref{theorem - ahlfors implies equih} they are equi-homogeneous.
\end{proof}
In fact, the open set condition in Theorem \ref{theorem - self-similar} may be relaxed to allow for the images of the attractor $F\subset \R^{d}$ to overlap in a limited number of ways. This `weak separation' condition, made precise in Fraser et al. \cite{Fraser2014}, is sufficient for $F$ to be Ahlfors-David regular provided that the attractor is not contained in a $d-1$ dimensional hyperplane (Theorem 2.1 of Fraser et al. \cite{Fraser2014}).

\subsection{Non-autonomous systems}\label{non-auto-sec}

We now consider pullback attractors of non-autonomous iterated function systems. Using the notation given in the introduction we consider a collection of index sets $\I_k\subset\N$ with ${\rm card}\left(\I_{k}\right)<\infty$ for all $k\in\N$. We also introduce some additional notation: define
\[
\I^0=\bigcup_{n=1}^\infty \I_1\times\cdots\times\I_n,
\]
for $\alpha=(i_1,\ldots,i_n)\in\I^0$ let
\[
f_\alpha=f_{i_1}\circ \cdots\circ f_{i_n}
\wwords{and} \sigma_\alpha=\sigma_{i_1}\cdots
\sigma_{i_n},
\]
and, if $n\ge 2$, we denote the truncation $(i_1,\ldots,i_{n-1})$ by $\alpha'$.

Before we give conditions ensuring the existence of a pullback attractor, we first show that the definition is sufficient to ensure uniqueness.

\begin{theorem}\label{unique}
The pullback attractor of a non-autonomous iterated
function system, if it exists, is unique.
\end{theorem}
\begin{proof}
Let $\set{F^{k}}$ and $\set{G^{k}}$ be pullback attractors of a non-autonomous iterated function system.
Given $k\in\N_0$ fixed, consider any point $x_k\in G^k$.
By property \ref{pullback attractor property 2} of pullback attractors there is $x_{k+1}\in G^{k+1}$ and $i_{k+1}\in\I_{k+1}$ such that $x_k=f_{i_{k+1}}(x_{k+1})$.
By induction define $x_j\in F^j$ and $i_j\in\I_j$ such that $x_j=f_{i_{j+1}}(x_{j+1})$ for all $j>k$.
Since the $x_j\in  G^j$ and the $\{G^j\}$ are uniformly bounded then $B=\{\, x_j : j> k\,\}$ is a bounded set.
Moreover, $x_k\in\S^{k,l}(B)$ for every $l>k$.
Consequently
\[
\rho_H(\{x_k\}, F^k)\le \rho_H(\S^{k,l}(B), F^k)\to 0 \words{as} l\to\infty
\]
by property \ref{pullback attractor property 3} of pullback attractors.
Since $F^k$ is closed, it follows that $x_k\in F^k$.  Therefore $G^k\subseteq F^k$.  Switching the roles of $ F^k$ and $ G^k$ yields that $ F^k= G^k$.
\end{proof}

We now show under natural conditions that the pullback attractor of a non-autonomous iterated function system exists.
For each $i\in\N$ let $b_i$ be the unique fixed point of $f_i$ such that $f_i(b_i)=b_i$.
\begin{theorem}\label{exist}
If
\[
	M=\sup\{\,|b_i|:i\in\N\,\}<\infty \wwords{and} \sigma^*=\sup\{\, \sigma_i:i\in\N\}<1,
\]
then the pullback attractor exists.
\end{theorem}
\begin{proof}
We first find a compact set $K$ such that for any bounded set $B\subset\R^d$ and any $k\in N$ there exists a corresponding value of $l(k,B)\ge k$ such that
\begin{equation}\label{cab}
	\S^{k,l}(B)\subseteq K \wwords{for all} l\ge l(k,B).
\end{equation}
Let $R= 2(1+\sigma^*) M/(1-\sigma^*)$ and $K$ be the closed ball of radius $R$ centred at the origin.
Consider a bounded set $B\subset\R^d$ with $\abs{x}\le L$ for all $x\in B$. If $\abs{x} \ge R$, then
\begin{align*}
	\abs{f_i(x)}&\le \abs{f_i(x)-f_i(b_i)}+\abs{b_i}
		\le \sigma^* \abs{x-b_i}+\abs{b_i}\\
		&\le \sigma^* \abs{x} + (1+\sigma^*) \abs{b_i}
		\le ((1+\sigma^*)/2) \abs{x}.
\end{align*}%
It follows that \eqref{cab} holds for $l(k,B)=k+\lceil\log(L/R)/\log(2/(1+\sigma^*))\rceil$.

Next we show that $\S^{k}$ is a contraction with respect to the Hausdorff metric.
Let $A,B\subset\R^d$ be compact. For each $a\in A$ choose $\pi(a)\in B$ such that
\[
	\abs{a-\pi(a)}=\min\{\, \abs{a-b} : b\in B\,\}.
\]
Any $x\in\S^{k}(A)$ may be written as $x=f_{i}(a)$
for some $a\in A$ and $i\in\I_k$.  Consequently
\[
	\rho_H(\{x\},\S^k(B))\le \abs{f_{i}(a)-f_{i}(\pi(a))}
		\le \sigma_i \abs{a-\pi(a)} \le \sigma^* \rho_H(A,B)
\]
implies
\begin{equation}\label{lcontract}
	\rho_H(\S^k(A),\S^k(B))\le \sigma^* \rho_H(A,B) \le \sigma^* \rho_H(A,B).
\end{equation}
Interchanging the roles of $A$ and $B$ yields
\[
	\rho_H(\S^k(A),\S^k(B))\le \sigma^* \rho_H(A,B).
\]

Given $k$ fixed, define $K_m=\S^{k,k+m}(K)$ for $m\in\N$ to obtain a sequence of compact subsets of $\R^d$.
If $m\le n$ then
\begin{align*}
	\rho_H(K_m,K_n)&=\rho_H(\S^{k,k+m}(K),\S^{k,k+m}\circ \S^{k+m,k+n}(K))\\
		&\le (\sigma^*)^m \rho_H(K,\S^{k+m,k+n}(K))
		\le (\sigma^*)^m 2R,
\end{align*}
which shows that $K_m$ is a Cauchy sequence. The completeness of the Hausdorff metric on the space of all compact subsets of $\R^d$ then yields a compact limit set, which we call $F^k$.

It remains to show that $ F^k$ satisfies the properties required of the pullback attractor.
Clearly $ F^k$ is compact.
Moreover, since $K_m\subseteq K$ for all $m$ then $F^k\subseteq K$ and so $ F^k$ is uniformly bounded.
To show the invariance property \ref{pullback attractor property 2} note that
\begin{align*}
	\rho_H( F^k&,\S^k( F^{k+1}))\\
	&\le \rho_H(F^k,\S^{k,k+m}(K)) + \rho_H(\S^k\circ\S^{k+1,k+m}(K),\S^k( F^{k+1}))\\
	&\le \rho_H(F^k,\S^{k,k+m}(K)) + \sigma^* \rho_H(\S^{k+1,k+m}(K), F^{k+1})\\
	&\to 0 \words{as} m\to\infty.
\end{align*}
Given $\varepsilon>0$ choose $m$ so large that $(\sigma^*)^m 2R<\varepsilon$.
Now choose $l\ge k+m$ so large that $\S^{k+m,l}(B)\subseteq K$. It follows that
\[
	\rho_H(\S^{k,l}(B), F^k) \le \rho_H(\S^{k,k+m}(K), F^k)=\rho_H(K_m, F^k) \le (\sigma^*)^m 2R <\varepsilon,
\]
and therefore property \ref{pullback attractor property 3} holds.
\end{proof}

We are now ready to formulate conditions on non-autonomous iterated function systems that guarantee that the resulting non-autonomous attractor is equi-homogeneous. We first prove that if the Moran open-set condition is satisfied then the pullback attractor is contained in the open sets, which is essentially the non-autonomous version of Falconer \cite{BkFalconer85} pp.122.

\begin{lemma}\label{finubar}
If a non-autonomous iterated function system satisfies the Moran open-set condition and the hypotheses of Theorem \ref{exist}, then $F^k\subseteq\overline{U^k}$.
\end{lemma}

\begin{proof}
Let $L$ be a uniform bound on $U^k$ and $F^k$ for all $k\in\N$. Since $\S^{k,l}(U^l)\subseteq U^l$ and $\S^{k,l}(F^l)=F^k$ for $k\le l$ then
\begin{align*}
	\rho_H(F^k,U^k)&
	\le \rho_H(S^{k,l}(F^l), S^{k,l}(U^l))
	\le (\sigma^*)^{l-k} \rho_H(F^l,U^l)
	\le (\sigma^*)^{l-k} 2L.
\end{align*}
Taking $l\to\infty$ it follows that $\rho_H(F^k,U^k)=0$ and consequently $F^k\subseteq\overline{U^k}$.
\end{proof}

We now consider the case where the contractions $f_i$ are similarities with contraction ratio $\sigma_{i}$. We begin by proving Theorem \ref{theorem - homogeneous case} from the Introduction, which treats the simplest situation where all the contraction ratios are the same at each level $k$ of the iteration. Note that this class of non-autonomous iterated function systems includes those whose pullback attractors are generalised Cantor sets (see Section which appear in Olson, Robinson \& Sharples \cite{ORS}.
Intuitively, each step of the iteration corresponds to a different scale, and since all the maps contract in the same way at that scale, it is natural that the pullback attractor is equi-homogeneous.

\mainone*

\begin{proof}
Since all the hypotheses are uniform in $k$ it is sufficient to show that $F^0$ is equi-homogeneous.
Let
\[
\pi_n={\textstyle\prod_{j=1}^{n}} c_j\wwords{and}
\eta=\sup\{\,{\rm
  diam}(\overline {U^k}): k\in\N\,\}.
\]
For $\delta\le c_1 \eta$ there exists $n\ge 2$ such that $\pi_n\eta < \delta \le \pi_{n-1}\eta.$
Define
\[
\J_n= \I_1\times\cdots\times\I_{n}.
\]
From the open-set condition we have
\begin{equation}\label{oscorm1}
  f_\alpha(U^n)\cap f_\beta(U^n)=\emptyset \words{for} \alpha,\beta\in \J_n \words{with} \alpha\ne\beta
\end{equation}
and by property 2 for pullback attractors, we obtain
\begin{equation}\label{attractcorm1}
  {\textstyle \bigcup_{\alpha\in \J_n}} f_\alpha(F^n) =\S^{0,n}(F^n)=F^0.
\end{equation}

We now use \eqref{oscorm1} and \eqref{attractcorm1} to show that $F^0$ is equi-homogeneous.  Let $x\in F^0$ be arbitrary.  Then $x\in f_{\alpha}(F^n)$ for some $\alpha\in \J_n$ and consequently
\[
{\rm diam}(f_{\alpha}(F^n)) = \pi_n {\rm diam}(F^n) \le
\pi_n\eta<\delta,
\]
which implies that $f_\alpha(F^n)\subseteq B_\delta(x)$.  It follows that
\[
B_\delta(x)\cap F^0 =B_\delta(x)\cap \bigcup_{\beta\in \J_n} f_\beta(F^n) \supseteq B_\delta(x)\cap f_{\alpha}(F^n)=f_\alpha(F^n).
\]
Therefore
\[
\cover(B_\delta(x)\cap F^0,\rho) \ge \cover(f_\alpha(F^n),\rho) = \cover(F^n,\rho/\pi_n)
\]
implies that
\begin{equation}\label{rhinfm1}
\inf_{x\in F^0} \cover(B_\delta(x)\cap F^0,\rho)
	\ge \cover(F^n,\rho/\pi_n).
\end{equation}

Let $A_n=\{\,\alpha\in \J_n : B_\delta(x)\cap f_\alpha(\overline {U^n})\ne\emptyset\,\}.$
Then $\beta\in A_{n-1}$ implies that
\[
f_\beta(\overline {U^{n-1}})\subseteq B_{\delta +{\rm diam} f_\beta(\overline {U^{n-1}})}(x) \subseteq B_{2\eta\pi_{n-1}}(x).
\]
Therefore by \eqref{oscorm1} we obtain
\begin{align*}
	\lambda(B_{2\eta\pi_{n-1}}(x))&\ge\lambda\Big(\bigcup_{\beta\in A_{n-1}} f_\beta(U^{n-1})\Big)=\sum_{\beta\in A_{n-1}}\lambda\big(f_\beta(U^{n-1})\big)\\
    &= {\rm card}(A_{n-1}) (\pi_{n-1})^d \lambda(U^{n-1}) \ge {\rm card}(A_{n-1})(\pi_{n-1})^d\epsilon_0.
\end{align*}%
Consequently
\[
	{\rm card}(A_n)\le {\rm card}(\I_n){\rm card}(A_{n-1})\le M
\]
where $M=(2\eta)^d N \lambda(B_1(0))/\epsilon_0$ is independent of $n$.
Therefore
\begin{align*}
\cover(B_\delta(x)\cap F^0,\rho)&\le\sum_{\alpha\in A_n} N\big(f_\alpha(F^n),\rho\big) =\sum_{\alpha\in A_n} \cover(F^n,\rho/\pi_n)\\
     &={\rm card}(A_n) \cover(F^n,\rho/\pi_n)\le M \cover(F^n,\rho/\pi_n).
\end{align*}
Taking the supremum of $F^0$ and combining this
with \eqref{rhinfm1} we obtain
\[
\sup_{x\in F^0}
	\cover(B_\delta(x)\cap F^0,\rho)\le M \cover(F^n,\rho/\pi_n) \le \inf_{x\in F^0} \cover(B_\delta(x)\cap F^0,\rho)
\]
which completes the proof of the theorem.
\end{proof}

\begin{corollary}\label{main2good}
If $\set{F^{k}}$ is the pullback attractor of a non-autonomous iterated function system of similarities satisfying the Moran open-set condition and $\sigma_i=c_k$ for all $i\in\I_k$, and all $k\in\mathbb{N}$ and $\inf\set{\sigma_{i} : i \in\mathbb{N}}=\sigma_{*}>0$ then each set $F^{k}$ is equi-homogeneous.
\end{corollary}
\begin{proof}
Let $B$ be a bounded set such that $U^k\subseteq B$ for every $k\in\N_0$.  The open-set condition then implies
\begin{align*}
      \lambda(B) &\ge \lambda(U^{k-1}) \ge \lambda({\textstyle \bigcup_{i\in\I_k}}f_i(U^k))\\
      &\ge {\textstyle \sum_{i\in\I_k}}\lambda(f_i(U^k)) \ge \sigma_*{\textstyle \sum_{i\in\I_k}}\lambda(U^k) \ge {\rm card}(I_k)\sigma_* \epsilon_0.
\end{align*}
Therefore ${\rm card}(\I_k)\le N$ where $N=\lambda(B)/(c_* \epsilon_0)$ and the result now follows from the application of Theorem \ref{theorem - homogeneous case}.
\end{proof}

We now examine the case when the $\sigma_i$ need not coincide for all the functions in each step of the iteration. This situation requires more refined analysis which we make easier by introducing the following notation.
Denote
\[
	\J^{k}={\textstyle \bigcup_{m\in\N}} \J^{k,k+m}
\wwords{where}
	\J^{k,n}=\I^{k+1}\times\cdots\times \I^{n}
\]
Given $\alpha\in\J^{k,n}$ denote $k_\alpha=k$ and $n_\alpha=n$. Note for $k_\alpha$ and $n_\alpha$ to be well defined, we assume as we may that $\I_k\cap\I_n=\emptyset$ for $k\ne n$. Our analysis will further make use of the set
\begin{equation}\label{Jdk}
	\J^k_\delta=\{\,\alpha\in\J^k: \sigma_\alpha\eta
		< \delta \le \sigma_{\alpha'}\eta\,\}.
\end{equation}
As in the proof of Theorem \ref{theorem - homogeneous case} we have the following facts:  if $\alpha,\beta\in\J^k_\delta$ then
\begin{align*}
	f_\alpha(U^{n_\alpha})\cap f_\beta(U^{n_\beta})&=\emptyset \wwords{when} \alpha\ne\beta
\intertext{and}
	{\textstyle \bigcup_{\alpha\in\J^k_\delta}} f_\alpha(F^{n_\alpha})&=F^k.
\end{align*}
We shall need, as before, an estimate on
the cardinality of
\[
	A^k_{\delta}=\{\,\alpha\in\J^k_\delta : f_\alpha(\overline{U^{n_\alpha}}) \cap B_{\delta}(x) \ne \emptyset \,\},
\]
which is provided by the following lemma.

\begin{lemma}  \label{Alemma}
If a non-autonomous iterated function system of similarities satisfies the Moran open-set condition and $\inf\set{\sigma_{i} : i \in\mathbb{N}}=\sigma_{*}>0$
then
\[
	{\rm card}(A^k_{\delta})\le \kappa_0,
\]
where $\kappa_0$ is independent of $k$ and $\delta$.
\end{lemma}
\begin{proof}
For $\alpha\in A^k_{\delta}$ we have
$f_\alpha(\overline{U^{n_\alpha}})\subseteq B_{2\delta}(x).$
Therefore
\begin{align*}
	\lambda(B_{2\delta}(x)) \ge \lambda\Big({\textstyle \bigcup_{\alpha\in A^k_{\delta}}} f_\alpha(U^{n_\alpha})\Big)
	&={\textstyle \sum_{\alpha\in A^k_{\delta}}}\lambda\big(f_\alpha(U^{n_\alpha})\big)
\\
	&\geq {\textstyle \sum_{\alpha\in A^k_{\delta}}}\sigma_{\alpha}^{d}
	\epsilon_{0}\\
	&\geq {\textstyle \sum_{\alpha\in A^k_{\delta}}}\left(\sigma_{\alpha^{\prime}}\sigma_{*}\right)^{d}
	\epsilon_{0}
	\geq {\rm card}(A^k_{\delta}) (\delta\sigma_*/\eta)^d\epsilon_0
\end{align*}
implies
\[
	{\rm card}(A^k_{\delta}) \le \lambda(B_1(0))(2\eta)/\sigma_*)^d/\epsilon_0 =\kappa_0
\]
where $\kappa_0= \lambda(B_1(0)) (2\eta/\sigma_*)^d/\epsilon_0$.
\end{proof}

We also need a bound on the cardinality of $\J^k_\delta$ (as defined in \eqref{Jdk}). In order to obtain this bound we make the uniformity assumption that for some $s>0$
\begin{equation}\label{hippohere}
	{\textstyle \sum_{i\in \I_k}} \sigma_i^s=1
	\wwords{for all} k\in\N
\end{equation}
(this was \eqref{hippo} in the Introduction) and consider a sequence of probability measures $\mu^k$ with support on $F^k$ defined by the relationships
\[
	\mu^k(f_i(B))= \mu^{k+1}(B)
	\frac{{\rm diam}(f_i(B))^s}{\sum_{j\in\I_{k+1}} {\rm diam}(f_j(B))^s}
	= \mu^{k+1}(B) \sigma_i^s
\]
where $k\in\N_0$, $i\in\I_{k+1}$ and $B$ is any Borel set. Note for $\alpha\in\J^{k,n}$ that
\[
	\mu^k(f_\alpha(F^n))=\mu_n(F^n)(\sigma_\alpha)^s=
		(\sigma_\alpha)^s.
\]
We are now ready to estimate ${\rm card}(\J^k_\delta)$.

\begin{lemma}\label{Jlemma}
If $\set{F^{k}}$ is the pullback attractor of a non-autonomous iterated function system of similarities satisfying the Moran open-set condition, $\inf\set{\sigma_{i} : i \in\mathbb{N}}=\sigma_{*}>0$ and \eqref{hippohere} then
\[
	\kappa_1 \delta^{-s}\le {\rm card}(\J^k_{\delta})\le
	\kappa_2 \delta^{-s},
\]
where $\kappa_1$ and $\kappa_2$ are independent of $k$ and $\delta$.
\end{lemma}
\begin{proof}
For the lower bound we estimate
\begin{align*}
	1=\mu^k(F^k)
	&=
	\mu^k
		\Big(
		{\textstyle \bigcup_{\alpha\in\J^k_{\delta}}}
		f_\alpha(F^{n_\alpha}) \Big)
		\le
		{\textstyle \sum_{\alpha\in\J^k_{\delta}}}
		\mu^k\big(
		f_\alpha(F^{n_\alpha}) \big)\\
		&=
		{\textstyle \sum_{\alpha\in\J^k_{\delta}}}
		(\sigma_\alpha)^s
		<
		{\textstyle \sum_{\alpha\in\J^k_{\delta}}}
		(\delta/\eta)^s
		= {\rm card}(\J^k_\delta) (\delta/\eta)^s.
\end{align*}
Therefore
\[
	{\rm card}(\J^k_\delta) >\kappa_1\delta^{-s}
\]
where $\kappa_1=\eta^s$.  For the upper bound we will use Lemma \ref{Alemma} to count the non-empty intersections where
\[
f_\alpha(F^{n_\alpha})\cap f_\beta(F^{n_\beta})\ne \emptyset \wwords{and} \beta\ne\alpha.
\]
We will do this inductively. Let $J_1=\J^k_\delta$ and pick $\alpha_1\in J_1$.
Define
\[
	J_{i+1}=\{\, \beta\in J_i :
		f_{\alpha_i}(F^{n_{\alpha_i}})
		\cap f_\beta(F^{n_\beta})= \emptyset
	\,\}.
\]
Since $f_{\alpha_k}(F^{n_{\alpha_k}})\subseteq B_\delta(x)$ for some $x$ it follows from Lemma \ref{Alemma} that
\[
	{\rm card}(J_{i+1})\ge {\rm card}(J_i) -
		{\rm card}(A^k_\delta)
		\ge {\rm card}(J_i)-\kappa_0
		\ge {\rm card}(J_1)-i\kappa_0.
\]
We can continue choosing $a_{i+1}\in J_{i+1}$ until $i=i_0$ where
\[
(i_0-1)\kappa_0< {\rm card}(J_1)\le i_0\kappa_0.
\]
By construction, it follows that
\[
	f_{\alpha_i}(F^{n_{\alpha_i}})\cap
	f_{\alpha_j}(F^{n_{\alpha_j}})=\emptyset
\wwords{for}
	i\ne j.
\]
Therefore
\begin{align*}
	1=\mu^k(F^k)
	&\ge \mu^k\Big(
		{\textstyle\bigcup_{i=1}^{i_0}} f_{\alpha_i}(F^{n_{\alpha_i}})
		\Big)
	=
		{\textstyle\sum_{i=1}^{i_0}} \mu^k\big(
			f_{\alpha_i}(F^{n_{\alpha_i}})\big)\\
	&=
		{\textstyle\sum_{i=1}^{i_0}} (\sigma_{\alpha_i})^s
	\ge
		{\textstyle\sum_{i=1}^{i_0}} (\sigma_* \delta/\eta)^s
	\ge  {\rm card}(\J^k_\delta) (\sigma_*\delta/\eta)^s/\kappa_0.
\end{align*}%
It follows that
\[
	{\rm card}(\J^k_\delta)\le \kappa_0 (\eta/\sigma_*)^s \delta^{-s}.
\]
Taking $\kappa_2=\kappa_0 (\eta/\sigma_*)^s$ finishes the proof.
\end{proof}

We are now ready to prove sufficient conditions for the equi-homogeneity of pullback attractors in the case when the contraction ratios need not coincide within each stage of the iteration. First we prove Theorem \ref{theorem - general} from the introduction, which has a uniformity assumption on the contraction ratios.

\mainthree*

\begin{proof}
We first estimate 
\[
\inf_{x\in F^0}
	\cover(F^0\cap B_\delta(x),\rho).
\]
Let $x\in F_0$ be arbitrary. There is an $\alpha\in \J^0_\delta$ such that $x\in f_\alpha(F^{n_\alpha})$ and consequently $F^0\cap B_\delta(x)\supseteq f_\alpha(F^{n_\alpha})$.
Define
\[
	\tilde A^k_\delta=\{\, \alpha\in \J^k_\delta :
		f_\alpha(\overline{U^{\alpha_k}})
	\cap B_{2\delta}(x)\ne \emptyset\,\}.
\]
Following the same proof as in Lemma \ref{Alemma} there exists $\tilde\kappa_0$ which is independent of $\delta$ and $k$ such that ${\rm card} (\tilde A^k_\delta)\le\tilde \kappa_0$.
Following the same proof as in Lemma \ref{Jlemma} we can find a sequence $\gamma_i\in \J^{n_\alpha}_{\rho/\sigma_\alpha}$ up to $i=\tilde \imath_0$ where
\[
	(\tilde \imath_0 -1)\tilde\kappa_0
	<{\rm card}(\J^{n_\alpha}_{\rho/\sigma_\alpha})
	\le \tilde \imath_0\tilde\kappa_0
\]
such that
$f_{\gamma_i}(F^{n_{\gamma_i}})\subseteq B_\delta(x_i)$ for $x_i\in f_{\gamma_i}(F^{n_{\gamma_i}})$ and
\[
B_{2\delta(x_i)}\cap F_{\gamma_j}(F^{n_{\gamma_j}})=\emptyset \wwords{for} i\ne j.
\]
In particular, we have found $x_i\in F^{n_\alpha}$ such that
\[
B_\delta(x_i)\cap B_\delta(x_j)=\emptyset
\wwords{for} i\ne j.
\]
It follows that
\begin{align*}
\pack(F^0\cap &B_\delta(x),\rho)
	\ge \pack(f_\alpha(F^{n_\alpha}),\rho)
	=\pack(F^{n_\alpha},\rho/\sigma_\alpha)
	\ge \tilde \imath_0\\
	&\ge {\rm card}(\J^{n_\alpha}_{\rho/\sigma_\alpha})/\tilde k_0
	\ge (\kappa_1/\tilde\kappa_0) (\sigma_\alpha/\rho)^s
	\ge (\kappa_1/\tilde\kappa_0)\sigma_*^s (\delta/\rho)^s
\end{align*}%
Therefore
\[
	\inf_{x\in F^0}
	\cover(F^0\cap B_\delta(x),\rho)\ge \kappa_3 (\delta/\rho)^s
\]
where $\kappa_3=\kappa_1 (2\sigma_*)^s/\tilde \kappa_0$.

We now estimate 
\[
\sup_{x\in F^0} \cover(F^0\cap B_\delta(x),\rho).
\]
  Let $x\in F^0$.  Applying Lemma~\ref{Jlemma} and Lemma~\ref{Alemma} we obtain
\begin{align*}
	\cover(F^0\cap B_\delta(x),\rho)
	&\le {\textstyle \sum_{\beta\in A^0_\delta}}
	\cover(f_\beta(F^{n_\beta}),\rho)
	={\textstyle \sum_{\beta\in A^0_\delta}}
	\cover(F^{n_\beta},\rho/\sigma_\beta)\\
	&={\textstyle \sum_{\beta\in A^0_\delta}}
	\cover\Big({\textstyle
		\bigcup_{\gamma\in\J^{n_\beta}_{\rho/\gamma_\beta}}}
		f_\gamma(F^{n_\gamma}),\rho/\sigma_\beta\Big)\\
	&\le {\textstyle \sum_{\beta\in A^0_\delta}}
{\textstyle \sum_{\gamma\in\J^{n_\beta}_{\rho/\gamma_\beta}}}
	\cover\big( F^{n_\gamma},\rho/(\sigma_\beta\sigma_\gamma)\big)\\
	&= {\textstyle \sum_{\beta\in A^0_\delta}}
{\textstyle \sum_{\gamma\in\J^{n_\beta}_{\rho/\gamma_\beta}}}
	\cover\big( F^{n_\gamma},\eta\big)\\
	&\le {\textstyle \sum_{\beta\in A^0_\delta}}
		{\rm card}(\J^{n_\beta}_{\rho/\gamma_\beta})
	\le {\textstyle \sum_{\beta\in A^0_\delta}}
		\kappa_2 (\sigma_\beta/\rho)^s\\
	&\le \kappa_0\kappa_2\eta^{-s} (\delta/\rho)^s.
\end{align*}
Taking $\kappa_4=\kappa_0\kappa_2\eta^{-s}$ and $M=\kappa_4/\kappa_3$ yields
\[
	\sup_{x\in F^0} \cover(F^0\cap B_\delta(x),\rho)
		\le \kappa_4(\delta/\rho)^s
		\le M
	\inf_{x\in F^0} \cover(F^0\cap B_\delta(x),\rho).
\]
We finish by noting that the above inequality also shows $\dim_{\rm A}F^k=s$.
\end{proof}

It is worth remarking that the proof of Theorem \ref{theorem - general}, in addition to proving that $F^k$ is equi-homogeneous, also shows that $F^k$ attains its upper and lower box-counting dimensions.
Thus, Theorem \ref{theorem - equi-homogeneous Assouad equal to box} implies that $\dim_{\rm B}F^k=\dim_{\rm A}F^{k}$ for all $k$.

For the final result in this section we note that the hypothesis \eqref{hippo} on $s$ in Theorem~\ref{theorem - general} can be weakened without changing the details of the proof, which is the content of Theorem \ref{theorem - general averaged} from the Introduction.

\mainfour*

Note that if $\sup\{\,\sigma_i : i\in\N\,\}=\sigma^{*}<1$ then an $s$ that satisfies \eqref{ahippo} is unique. Indeed, suppose $s_0$ satisfies \eqref{ahippo} and $\delta\ne 0$.
If $\delta>0$ then
\[
	{\textstyle\sum_{\alpha\in\J_{k,k+n}}}
\sigma_\alpha^{s_0+\delta}
		\le
(\sigma^*)^{\delta n}
	{\textstyle\sum_{\alpha\in\J_{k,k+n}}}
\sigma_\alpha^{s_0}
	\le (\sigma^*)^{\delta n} L\to 0
\words{as}
	n\to\infty
\]
shows that the lower bound in \eqref{ahippo} could not hold for $s=s_0+\delta$.  On the other hand, if $\delta<0$ then
\[
	{\textstyle\sum_{\alpha\in\J_{k,k+n}}}
\sigma_\alpha^{s_0+\delta}
		\ge
(\sigma^*)^{\delta n}
	{\textstyle\sum_{\alpha\in\J_{k,k+n}}}
\sigma_\alpha^{s_0}
	\ge (\sigma^*)^{\delta n} L\to \infty
\words{as}
	n\to\infty
\]
shown that the upper bound could not hold.  We conclude that there is at most one value for $s$ such that \eqref{ahippo} holds.

The bounds \eqref{ahippo} essentially state that \eqref{hippo} holds uniformly when averaged over long enough
sequences of iterations.

\section{Conclusion}
We have demonstrated that the equi-homogeneous sets include a large class of attractors of iterated functions systems, both autonomous and non-autonomous, in addition to the generalised Cantor sets and homogeneous Moran sets considered in Olson, Robinson and Sharples \cite{ORS}. Further, as equi-homogeneous sets have identical dimensional detail at all points at each fixed length scale, we have shown that the calculation of their Assouad dimensions can be much simplified. Finally, we have demonstrated that equi-homogeneity is independent of any previously defined notion of dimensional equivalence, establishing equi-homogeneity as a novel and useful tool in the analysis of fractal sets.

\bibliographystyle{abbrv}
\bibliography{references.bib}

\begin{thebibliography}{10}

\bibitem{Assouad}
P.~Assouad.
\newblock {{\'{E}}tude d'une dimension m{\'{e}}trique li{\'{e}}ea la
  possibilit{\'{e}} de plongements dans $\mathbb{R}^{n}$}.
\newblock {\em C.R. Acad. Sci. Paris S{\'{e}}r. A-B}, 288(15):A731--A734, 1979.

\bibitem{Bouligand}
G.~Bouligand.
\newblock {Ensembles impropres et nombre dimensionnel}.
\newblock {\em Bull. des Sci. Math{\'{e}}matiques}, 52:320--344, 1928.

\bibitem{Carvalho2013}
A.~N. Carvalho, J.~A. Langa, and J.~C. Robinson.
\newblock {\em {Attractors for infinite-dimensional non-autonomous dynamical
  systems}}, volume 182 of {\em Applied Mathematical Sciences}.
\newblock Springer, New York, 2013.

\bibitem{Cheban2002}
D.~Cheban, P.~Kloeden, and B.~Schmalfu{\ss}.
\newblock {The relationship between pullback, forwards and global attractors of
  nonautonomous dynamical systems}.
\newblock {\em Nonlinear Dyn. Syst. Theory}, 2:125--144, 2002.

\bibitem{EFNT}
A.~Eden, C.~Foias, B.~Nicolaenko, and R.~Temam.
\newblock {\em {Exponential attractors for dissipative evolution equations}},
  volume~37 of {\em Research in Applied Mathematics}.
\newblock Masson, Paris and John Wiley {\&} Sons, Chichester, 1994.

\bibitem{BkFalconer85}
K.~J. Falconer.
\newblock {\em {The geometry of fractal sets}}, volume~85 of {\em Cambridge
  Tracts in Mathematics}.
\newblock Cambridge University Press, 1986.

\bibitem{BkFalconer14}
K.~J. Falconer.
\newblock {\em {Fractal geometry: mathematical foundations and applications}}.
\newblock John Wiley {\&} Sons, third edition, 2014.

\bibitem{FarkasFraser14}
{\'{A}}.~Farkas and J.~Fraser.
\newblock {On the equality of Hausdorff measure and Hausdorff content}.
\newblock {\em J. Fractal Geom.}, 2(4):403--429, 2015.

\bibitem{Fraser2013}
J.~M. Fraser.
\newblock {Assouad type dimensions and homogeneity of fractals}.
\newblock {\em Trans. Am. Math. Soc.}, 336:6687--6733, 2014.

\bibitem{Fraser2014}
J.~M. Fraser, A.~M. Henderson, E.~J. Olson, and J.~C. Robinson.
\newblock {On the Assouad dimension of self-similar sets with overlaps}.
\newblock {\em Adv. Math. (N. Y).}, 273:188--214, 2015.

\bibitem{Heinonen}
J.~Heinonen.
\newblock {\em {Lectures on analysis on metric spaces}}.
\newblock Springer-Verlag, New York, 2001.

\bibitem{Hua2000}
S.~Hua, H.~Rao, Z.~Wen, and J.~Wu.
\newblock {On the structures and dimensions of Moran sets}.
\newblock {\em Sci. China Ser. A Math.}, 43(8):836--852, 2000.

\bibitem{Hutch}
J.~E. Hutchinson.
\newblock {Fractals and self-similarity}.
\newblock {\em Indiana Univ. Math. J.}, 30(5):713--747, 1981.

\bibitem{Kloedendifference}
P.~E. Kloeden.
\newblock {Pullback attractors in nonautonomous difference equations}.
\newblock {\em J. Differ. Equations Appl.}, 6(1):33--52, 2000.

\bibitem{Kloedensemidynamical}
P.~E. Kloeden.
\newblock {Pullback attractors of nonautonomous semidynamical systems}.
\newblock {\em Stochastics Dyn.}, 3(01):101--112, 2003.

\bibitem{KR}
P.~E. Kloeden and M.~Rasmussen.
\newblock {\em {Nonautonomous dynamical systems}}, volume 176 of {\em
  Mathematical Surveys and Monographs}.
\newblock American Mathematical Soc., 2011.

\bibitem{KloedenStonier1998}
P.~E. Kloeden and D.~J. Stonier.
\newblock {Cocycle attractors in nonautonomously perturbed differential
  equations}.
\newblock {\em Dyn. Contin. Discret. Implusive Syst.}, 4(2):211--226, 1998.

\bibitem{Larman}
D.~G. Larman.
\newblock {A new theory of dimension}.
\newblock {\em Proc. London Math. Soc.}, s3-17(1):178--192, 1967.

\bibitem{Li2013}
J.~Li.
\newblock {Assouad dimensions of Moran sets}.
\newblock {\em C.R. Acad. Sci. Paris S{\'{e}}r. I}, 351(1-2):19--22, 2013.

\bibitem{Luuk}
J.~Luukkainen.
\newblock {Assouad dimension: Antifractal metrization, porous sets, and
  homogeneous measures}.
\newblock {\em J. Korean Math. Soc.}, 35(1):23--76, 1998.

\bibitem{MikeTyson}
J.~Mackay and J.~Tyson.
\newblock {\em {Conformal Dimension}}, volume~54 of {\em University Lecture
  Series}.
\newblock American Mathematical Society, Providence, Rhode Island, 2010.

\bibitem{Mandelbrot75}
B.~B. Mandelbrot.
\newblock {\em {Les Objets Fractals - Forme, Hasard et Dimension}}.
\newblock Flammarion, Paris, 1975.

\bibitem{Moran}
P.~A.~P. Moran.
\newblock {Additive functions of intervals and Hausdorff measure}.
\newblock {\em Math. Proc. Cambridge Philos. Soc.}, 42(01):15--23, 1946.

\bibitem{EJO}
E.~J. Olson.
\newblock {Bouligand dimension and almost Lipschitz embeddings}.
\newblock {\em Pacific J. Math.}, 202(2):459--474, 2002.

\bibitem{ORSMoransets}
E.~J. Olson, J.~C. Robinson, and N.~Sharples.
\newblock {Equi-homogeneity of Moran sets}.
\newblock {\em (in preparation)}.

\bibitem{ORS}
E.~J. Olson, J.~C. Robinson, and N.~Sharples.
\newblock {Generalised Cantor sets and the dimension of products}.
\newblock {\em Math. Proc. Cambridge Philos. Soc.}, 160(1):51--75, 2016.

\bibitem{JCR}
J.~C. Robinson.
\newblock {\em {Dimensions, embeddings, and attractors}}, volume 186 of {\em
  Cambridge Tracts in Mathematics}.
\newblock Cambridge University Press, 2011.

\bibitem{RobinsonSharples13RAEX}
J.~C. Robinson and N.~Sharples.
\newblock {Strict inequality in the box-counting dimension product formulas}.
\newblock {\em Real Anal. Exch.}, 38(1):95--120, 2013.

\bibitem{Schmalfuss1992}
B.~Schmalfuss.
\newblock {Backward cocycles and attractors of stochastic differential
  equations}.
\newblock In {\em International Seminar on Applied Mathematics-Nonlinear
  Dynamics: Attractor Approximation and Global Behaviour}, pages 185--192,
  Dresden, Germany, 1992.

\bibitem{Tyson08}
J.~T. Tyson.
\newblock {Global conformal Assouad dimension in the Heisenberg group}.
\newblock {\em Conform. Geom. Dyn.}, 12(04):32--58, 2008.

\bibitem{Wen2001}
Z.-Y. Wen.
\newblock {Moran sets and Moran classes}.
\newblock {\em Chinese Sci. Bull.}, 46(22):1849--1856, 2001.

\end{thebibliography}

\end{document}